\documentclass[12pt,reqno]{amsart}

\usepackage[textheight=22cm,textwidth=16.6cm,headsep=1cm,centering]{geometry}

\usepackage[english]{babel}
\usepackage[utf8]{inputenc}
\usepackage{amsmath, amsthm, amssymb}
\usepackage{amsfonts}
\usepackage{amsthm}
\usepackage{mathrsfs}
\usepackage[non-sorted-cites]{amsrefs}

\usepackage{xcolor}

\usepackage[titletoc]{appendix}

\usepackage{hyperref}

\makeatletter	
\ifcsname phantomsection\endcsname
\newcommand*{\qrr@gobblenexttocentry}[5]{}
\else
\newcommand*{\qrr@gobblenexttocentry}[4]{}
\fi
\newcommand*{\addsection}{
	\addtocontents{toc}{\protect\qrr@gobblenexttocentry}
	\section}
\makeatother

\theoremstyle{plain}
\newtheorem{theorem}{Theorem}[section]
\newtheorem{corollary}[theorem]{Corollary}
\newtheorem{lemma}[theorem]{Lemma}
\newtheorem{proposition}[theorem]{Proposition}

\theoremstyle{definition}

\newtheorem{remark}[theorem]{Remark}

\numberwithin{equation}{section}

\newcommand{\R}{\mathbb{R}}
\newcommand{\e}{\varepsilon}
\newcommand{\abs}[1]{\left\lvert#1\right\rvert}
\newcommand{\plaplacian}{\Delta_p}
\newcommand{\norm}[1]{\|{#1}\|}
\newcommand{\seminorm}[1]{\left [{#1} \right ]}
\newcommand{\diver}{\textnormal{div}}

\newcommand{\average}{{\mathchoice {\kern1ex\vcenter{\hrule height.4pt
				width 6pt depth0pt} \kern-9.7pt} {\kern1ex\vcenter{\hrule
				height.4pt width 4.3pt depth0pt} \kern-7pt} {} {} }}
\newcommand{\ave}{\average\int}

\begin{document}
	
\title[Optimal regularity of stable solutions to $p$-Laplace equations]{Optimal regularity of stable solutions to nonlinear equations involving the $p$-Laplacian}

\author[X. Cabr\'e]{Xavier Cabr\'e}
\author[P. Miraglio]{Pietro Miraglio}
\author[M. Sanch\'on]{Manel Sanch\'on}

\address{X.C.\textsuperscript{1,2}
\textsuperscript{1} ICREA, Pg.\ Lluis Companys 23, 08010 Barcelona, Spain
\&
\textsuperscript{2} Universitat Polit\`ecnica de Catalunya, Departament de Matem\`{a}tiques, Diagonal 647, 08028 Barcelona, Spain
}
\email{xavier.cabre@upc.edu}

\address{P.M.\textsuperscript{1,2},
	\textsuperscript{1} Universit\`a degli Studi di Milano, Dipartimento di Matematica, via Cesare Saldini 50, 20133 Milano, Italy
	\&
	\textsuperscript{2} Universitat Polit\`ecnica de Catalunya, Departament
	de Matem\`{a}tiques, Diagonal 647,
	08028 Barcelona, Spain}
\email{miraglio.pietro@gmail.com}

\address{M.S.,
	Grupo AIA - Aplicaciones en Inform\'atica Avanzada, SL, ESADECREAPOLIS
	Planta 2a Bloc C Portal 1,
	Av. Torre Blanca 57, CP-08172 Sant Cugat del Vall\`es, Spain}
\email{manel.sanchon@gmail.com}


\thanks{X.C. and P.M. are supported by grant MTM2017-84214-C2-1-P (Government of Spain) and are members of the research group 2017SGR1392 (Government of Catalonia).
}

\begin{abstract}
	We consider the equation $-\Delta_p u=f(u)$ in a smooth bounded domain of $\mathbb{R}^n $, where $\Delta_p$ is the $p$-Laplace operator. Explicit examples of unbounded stable energy solutions are known if $n\geq p+4p/(p-1)$.
	Instead, when $n<p+4p/(p-1)$, stable solutions have been proved to be bounded only in the radial case or under strong assumptions on $f$.
	
	In this article we solve a long-standing open problem: we prove an interior $C^\alpha$ bound for stable solutions which holds for every nonnegative $f\in C^1$ whenever $p\geq2$ and the optimal condition $n<p+4p/(p-1)$ holds.
	When $p\in(1,2)$, we obtain the same result under the non-sharp assumption~$n<5p$.
	These interior estimates lead to the boundedness of stable and extremal solutions to the associated Dirichlet problem when the domain is strictly convex.
	
	Our work extends to the $p$-Laplacian some of the recent results of Figalli, Ros-Oton, Serra, and the first author for the classical Laplacian, which have established the regularity of stable solutions when $p=2$ in the optimal range $n<10$.
\end{abstract}

\maketitle

\section{Introduction}
Given $ p\in(1,+\infty)$, a function $ f\in C^1(\R)$, and a smooth bounded domain $ \Omega\subset\R^n $, we consider the elliptic equation involving the $p$-Laplacian
\begin{equation}\label{plap_eq}
-\plaplacian u=-\text{div}(\abs{\nabla u}^{p-2}\nabla u)= f(u) \qquad \text{in}\,\,\Omega,
\end{equation}
as well as the Dirichlet problem
\begin{equation}
\label{plap_Dir}
\left\{
\begin{array}{rcll}
-\plaplacian u&=& f(u) & \qquad \text{in}\,\,\Omega \\
u&>&0 & \qquad\text{in}\,\,\Omega\\
u&=&0 & \qquad\text{on}\,\,\partial\Omega.\\
\end{array}\right.
\end{equation}

A function $u\in W^{1,p}(\Omega)$ is an \textit{energy solution} to equation \eqref{plap_eq} if $f(u)\in L^1_{\rm loc}(\Omega)$ and it satisfies
\begin{equation*}\label{plap_weak}
\int_\Omega\abs{\nabla u}^{p-2}\nabla u\cdot\nabla\xi\,dx=\int_\Omega f(u)\xi\,dx \quad \quad \text{for every}\,\,\,\xi\in C_c^1(\Omega).
\end{equation*}
We say that $u$ is a \textit{regular solution} to \eqref{plap_eq} if it is an energy solution and, in addition, it satisfies $f(u)\in L^\infty(\Omega)$.
By classical regularity results every regular solution to \eqref{plap_eq} is $ C^{1,\vartheta}(\Omega)$ for some $ \vartheta>0 $, while regular solutions to \eqref{plap_Dir} are  $C^{1,\vartheta}(\overline{\Omega})$ --- see \cite{DiB,T,L}. This is the best regularity that one can expect for regular solutions to nonlinear equations involving the $p$-Laplacian.

Our interest lies in nonlinearities $f(u)$ that grow faster than $u^{p-1}$ as $u\to+\infty$. In such case, it is easy to see that the energy functional associated to the equation,
\[
E(u)=\int_\Omega\left(\frac1p \abs{\nabla u}^p - F(u)\right)\,dx,
\]
where $F'=f$, admits no absolute minimizer. However, in some important cases, the functional will admit local minimizers, or more generally stable solutions --- as defined next. For instance, they will exist, for a certain range of parameters $\lambda$, in the case of Gelfand-type problems, in which the nonlinearity is given by $\lambda f(u)$. We will apply our results to these well studied problems, whose model nonlinearities are $\lambda e^u$ and $\lambda (1+u)^m$ with $m>p-1$.

To define the notion of stable solutions to \eqref{plap_eq} we need a function space introduced by Damascelli and Sciunzi in~\cite{DS} and extensively used in~\cite{FSV,FSV1,CasSa}.
Assuming that $u$ is a regular solution, we define the weighted Sobolev space $W^{1,2}_{\sigma}(\Omega)$ with weight 
\[\sigma=\abs{\nabla u}^{p-2}\]
as the completion of $C^1(\overline{\Omega})$ with respect to the norm
\begin{equation}\label{weighted_norm}
\begin{split}
\norm{\xi}_{W^{1,2}_\sigma(\Omega)}:&=\norm{\xi}_{L^2(\Omega)}+\norm{\nabla\xi}_{L_\sigma^2(\Omega)}
\\
&=\left(\int_\Omega\xi^2\,dx\right)^\frac12+\left(\int_\Omega\abs{\nabla u}^{p-2}\abs{\nabla\xi}^2\,dx\right)^\frac12.
\end{split}
\end{equation}
We also define $W^{1,2}_{\sigma,0}(\Omega)$ as the completion of $C^1_c(\Omega)$ with respect to the $W^{1,2}_\sigma$-norm.

As in~\cite{DS,FSV,FSV1,CasSa}, a regular solution $u$ to \eqref{plap_eq} is said to be {\it stable} if the second variation of the energy functional at $u$ is nonnegative definite, i.e.,
\begin{equation}\label{stability}
\int_{\Omega}\left\{\abs{\nabla u}^{p-2}\abs{\nabla\xi}^2+(p-2)\abs{\nabla u}^{p-4}\left(\nabla u\cdot\nabla\xi\right)^2\right\}\,dx-\int_\Omega f'(u)\xi^2\,dx\geq0
\end{equation}
for every $ \xi\in \mathcal{T}_u $, where $\mathcal{T}_u$ is the space of test functions defined 
as
\begin{equation*}
\mathcal{T}_u:=
\begin{aligned}
\begin{cases}
W^{1,2}_{\sigma,0}(\Omega)\qquad&\text{if}\,\,p\geq2,
\\
\{\xi\in W^{1,2}_0(\Omega):\,\,\norm{\nabla\xi}_{L_\sigma^2(\Omega)}<\infty	\} \qquad&\text{if}\,\,p\in(1,2).
\end{cases}
\end{aligned}
\end{equation*}
This class of functions, introduced in the seminal paper~\cite{DS} by 
Damascelli and Sciunzi, is a Hilbert space\footnote{The distinction between $p\geq 2$ and $p\in (1,2)$ in the definition is made to guarantee that $\mathcal{T}_u$ is a complete space; see \cite{DS,CasSa} for more details.} which is natural here, in the sense that it is a completion of smooth functions under the main quantity in the first integral in the stability inequality~\eqref{stability}. In addition, the space will contain the test function $\xi=\abs{\nabla u}\eta$, with $\eta$ a cut off function, thanks to the important estimate~\eqref{firstDS} proved in~\cite{DS}.  This test function was used by Farina, Sciunzi, and Valdinoci~\cite{FSV, FSV1} to derive Theorem~\ref{thm_SZ_stability} below, 
a result that will play a key role in our proof.

We point out that, as we will see in Remark \ref{rmk_integrability},
the first integral in \eqref{stability} can also be computed over $\{\abs{\nabla u}>0\}\cap\Omega$ instead of $\Omega$ without altering the value of the integral.

Stable solutions to \eqref{plap_Dir} have been extensively investigated in the case $p=2$, starting in 1975 with the seminal paper of Crandall and Rabinowitz \cite{CR} --- see the monograph \cite{D} by Dupaigne and the survey \cite{C2} for a description of the classical literature in the field. As we will see, regularity theory for stable solutions to semilinear equations has been an important line of research in the last decades.
Very recently Figalli, Ros-Oton, Serra, and the first author \cite{CFRS}
provided a complete answer to this topic, proving that stable energy solutions to \eqref{plap_Dir} with $p=2$ are bounded --- and therefore smooth --- up to dimension 9, whenever the nonlinearity $f$ is nonnegative, nondecreasing, and convex\footnote{These hypotheses on the nonlinearity are related to some open problems stated by Brezis \cite{B} and by Brezis-Vázquez \cite{BV} about \textit{extremal solutions}, as we shall explain in Subsection~\ref{subsec_extremal}.} For interior boundedness, \cite{CFRS} establishes that the nonnegativeness of $f$  suffices. These results are indeed optimal, since examples of singular stable energy solutions to \eqref{plap_Dir} with $p=2$ are known for such nonlinearities when $n\geq10$. 

For $p>1$ the study of the boundedness of stable solutions to \eqref{plap_Dir} was initiated by Garc\'ia-Azorero, Peral, and Puel \cite{GP,GPP} in the case $f(u)=\lambda e^u$. They established that, with this choice of the nonlinearity, stable energy solutions are bounded\footnote{To be totally precise, these two references, as well as others cited next, focus on the boundedness of the so-called {\it extremal solution} --- a  type of stable solution that we introduce later in Subsection~\ref{subsec_extremal}. However, the proofs also establish an apriori $L^\infty$ bound for the more general class of stable regular solutions.} whenever 
\begin{equation}\label{optimal_dimension}
n<p+\frac{4p}{p-1},
\end{equation}
proving also that this condition is optimal. Indeed, they showed that when $\Omega=B_1$ and $ n\geq p+4p/(p-1) $, then $ u(x)=\log(\abs{x}^{-p}) $ is a singular stable energy solution of~\eqref{plap_Dir} with $f(u)=p^{p-1}(n-p)e^u$.

The boundedness of stable solutions to \eqref{plap_Dir} in the optimal dimension range \eqref{optimal_dimension} has been proved also for power-like nonlinearities\footnote{More generally, the boundedness of stable solutions in the sharp dimension range \eqref{optimal_dimension} is also ensured if $f\in C^2$ satisfies \eqref{f_p_extremal_hp}, \eqref{f_p_convexity}, and that the limit $\tau:=\lim_{t\rightarrow+\infty} f(t)f''(t)/f'(t)^2$ exists. Notice that this last assumption is rather strong. In addition, one can check that if the limit exists then necessarily $\tau\leq 1$, otherwise $f$ blows-up ``in finite time'' and is not defined in the whole real line. When $\tau=1$, the result in the optimal dimension range follows from \cite{San}. Instead, if $\tau<1$, then $f(t)\leq C(1+t)^m$ for some~$m$ and the result follows from \cite{CSa}. 
For the classical case $p=2$, this was proven in \cite{CR}.}  by the first and third authors~\cite{CSa} and in the radial case by Capella, the first and third authors~\cite{CCaSa} for every locally Lipschitz nonlinearity.

Instead, in the nonradial case and for general nonlinearities $f\in C^1$, only some nonoptimal boundedness results have been proved.
Some common assumptions on the nonlinearity are
\begin{equation}\label{f_p_extremal_hp}
f(0)>0, \qquad f'>0, \qquad \lim_{t\to+\infty}\frac{f(t)}{t^{p-1}}=+\infty,
\end{equation}
as well as
\begin{equation}\label{f_p_convexity}
\begin{cases}
(f(t)-f(0))^\frac{1}{p-1} \,\,\, \text{is convex}\qquad &\text{if}\,\,\, p\geq2,
\\
f \,\,\, \text{is convex}\qquad &\text{if}\,\,\, p\in(1,2).
\end{cases}
\end{equation}
Under these assumptions on $f$, the third author \cite{San,San1} extended Nedev's approach \cite{N} to the semilinear ($p=2$) problem, proving that stable solutions are bounded whenever
\begin{equation}\label{intro-Ned-San}
\begin{cases}
n<p+\frac{p}{p-1}\qquad &\text{if}\,\,\, p\geq2,
\\
n\leq p+\frac{2p}{p-1}(1+\sqrt{2-p})\qquad &\text{if}\,\,\, p\in(1,2).
\end{cases}
\end{equation}
Note that these two bounds on $n$ are strictly below the optimal one, \eqref{optimal_dimension}. The result for $p\geq2$ was proved in \cite[Theorem 1]{San1}, while the one for $p<2$ was obtained in \cite[Theorem~1]{San} --- note here that $(f(t)-f(0))^\frac{1}{p-1}$ is convex since $p\in(1,2)$, $f'>0$, and $f''\geq0$.

Some years later, Castorina and the third author~\cite{CasSa} extended to the case of the $p$-Laplacian the ideas of the first author~\cite{C} for $p=2$, where the boundedness of stable solutions to semilinear equations was proved whenever $n\leq4$ and $\Omega$ is convex.
Following this strategy, \cite{CasSa} established the boundedness of stable solutions for strictly convex domains in dimension $n\leq p+2$ for every~$p>2$, under hypotheses~\eqref{f_p_extremal_hp} and~\eqref{f_p_convexity} on $f$. Note that $p+\frac{p}{p-1}\leq p+2$ when $p\geq2$, but that the result from \cite{CasSa} only applies to strictly convex domains --- in contrast with the one of \cite{San1} for general smooth domains.

When $p> 2$ and $n\geq4$, the condition $n\leq p+2$ from \cite{CasSa} has been recently improved by the second author \cite{Miraglio}. This work extends to the $p$-Laplacian the method of the first author~\cite{C1} for $p=2$ and establishes boundedness in strictly convex domains under the assumption
\begin{equation}\label{dimMiraglio}
n < \frac12 \left(\sqrt{(p-1)(p+7)}+p+5	\right)
\end{equation}
for $p>1$. When $p>2$ this condition is better than the one of \cite{CasSa}.

The aim of the present paper is to extend the interior regularity results of Figalli, Ros-Oton, Serra, and the first author~\cite{CFRS} for $p=2$ to the case of the $p$-Laplacian. More precisely, we obtain an interior $C^\alpha$ apriori estimate for stable regular solutions to~\eqref{plap_eq}, with a linear control of the $C^\alpha$ norm in terms only of the $L^1$ norm and with a constant independent of the nonlinearity. 
Our result holds for every nonnegative\footnote{Our proof uses the nonnegativeness of $f$ to control several times, in Section \ref{sec_higher}, weighted integrals of $|\Delta_{p} u|=-\Delta_{p} u$ through integration by parts; see Lemma~\ref{positive_plap} in this respect. That the solution $u$ is $p$-superharmonic is also used crucially in the proof of Lemma~\ref{lemma_doubling}, which controls the $L^{p}$-norm of its full gradient by a weighted $L^{2}$-norm of the radial derivative; here we use a compactness argument that requires $p$-superharmonicity. However, as pointed out in more detail in Subsection~\ref{subsec_open}, it is not known if our boundedness results hold without assuming the nonnegativeness of the nonlinearity.
All these comments apply also to the case $p=2$ treated in \cite{CFRS}.
}
$f\in C^1$
in the optimal dimension range~\eqref{optimal_dimension} when $ p\geq2 $ and in the range $n<5p$ when $p\in(1,2)$.
Under these hypotheses on $n$ and $p$ and as a direct consequence of our estimates, we show the boundedness of stable solutions to the Dirichlet problem~\eqref{plap_Dir} in strictly convex domains.

\subsection{Main results}
Theorem \ref{thm_interior} below is the main result of the present paper. It establishes a universal interior bound on the $ C^\alpha $ norm of stable solutions to \eqref{plap_eq} assuming that the nonlinearity is nonnegative and that $n$ and $p$ satisfy condition~\eqref{dimension} below. When $p\geq 2$ this condition on the dimension is optimal\footnote{After the statement of Theorem \ref{thm_extremal_new} on extremal solutions, we will compare our nonoptimal condition $n < 5p$ for $p\in (1,2)$ with previously known results.}.
Under the same assumption on the nonlinearity but in every dimension, we also prove a higher integrability result for the gradient of stable solutions to \eqref{plap_eq}. This last result relies on an $L^1$ control on the $(p+1)$-Laplacian $\Delta_{p+1}$ of a stable solution, a new estimate that we also state in the theorem. Since the results are local, we state them in the unit ball $ B_1\subset\R^n $.

\vspace{7mm}

\begin{theorem}\label{thm_interior}
	Let $B_1$ be the unit ball of $\R^n$ and $u$ a stable regular solution of 
	\[-\Delta_pu=f(u)\qquad\text{in}\,\,B_1,\]
	with $ f\in C^1(\R)$ nonnegative.
	
	Then,
	\begin{equation}\label{L1controls_p}
	\norm{\nabla u}_{L^{p+\gamma}(B_{1/2})}\leq C\norm{u}_{L^1(B_1)}
	\end{equation}
	and
	\begin{equation}\label{pplus1}
	\norm{\Delta_{p+1}u}_{L^1(B_{1/2})}\leq C \norm{u}_{L^1(B_1)},
	\end{equation}
	where $\gamma$ and $C$ are positive constants depending only on $n$ and $p$. In addition, if
	\begin{equation}\label{dimension}
	\begin{cases}
	n<p+\frac{4p}{p-1}\qquad &\text{for}\quad p\geq2, \\
	n<5p\qquad&\text{for}\quad p\in(1,2),
	\end{cases}
	\end{equation}	
	then
	\begin{equation}\label{C.alfa}
	\norm{u}_{C^{\alpha}(\overline{B}_{1/2})}\leq C\norm{u}_{L^1(B_1)},
	\end{equation} 
	where $\alpha$ and $C$ are positive constants depending only on $n$ and $p$.
\end{theorem}

Combining the interior bounds of the previous theorem with known boundary estimates in strictly convex domains, we obtain global bounds for stable solutions to \eqref{plap_Dir} in such domains. The boundary estimates are proved in~\cite[Proposition 3.1]{CasSa} assuming that the nonlinearity is positive.

\begin{corollary}\label{cor_convex}
	Let $ \Omega\subset\R^n $ be a smooth bounded domain which is strictly convex and $u$ a stable regular solution of the Dirichlet problem \eqref{plap_Dir}, with $f\in C^1(\R)$ positive. 
	
	Then,
	\begin{equation}\label{Lp_convex_bound}
	\norm{\nabla u}_{L^p(\Omega)}\leq C,
	\end{equation}
	where $C$ depends only on $\Omega$, $p$, $f$, and $\norm{u}_{L^1(\Omega)}$.
	If in addition $n$ and $p$ satisfy \eqref{dimension}, then
	\begin{equation}\label{convex_bound}
	\norm{u}_{L^\infty(\Omega)}\leq C\norm{u}_{L^1(\Omega)},
	\end{equation}
	where $C$ depends only on $\Omega$ and $p$.
\end{corollary}

In Subsection \ref{subsec_ideas} we describe the main ideas in the proof of the H\"older bound in Theorem~\ref{thm_interior}. In short, it is based on combining two main ingredients. The first one is the key estimate of Lemma~\ref{lemma_stability}, based on choosing an appropriate test function~$\xi$, given by \eqref{first-test} and where $r=|x|$, in the stability condition~\eqref{stability}. In comparison with the case $p=2$ in \cite{CFRS}, the computations to arrive to Lemma~\ref{lemma_stability} are considerably longer and involve a delicate point which does not occur for $p=2$:  to avoid the critical set $\{|\nabla u|=0\}$ where the $p$-Laplacian degenerates, we need to include an additional cut-off factor $\phi(|\nabla u|/\varepsilon)$ within the test function $\xi$ in \eqref{first-test}.

The second ingredient is Lemma~\ref{lemma_doubling}, a $p$-Laplacian analogue of the corresponding result in~\cite[Lemma 3.1]{CFRS}.
This crucial lemma, which is proved through a compactness argument, allows to control the $L^p$ norm of the full gradient of the solution by the integral of its weighted squared radial derivative. In \cite{CFRS} its proof relied on the higher integrability estimate~\eqref{L1controls_p} (for $p=2$) and on an additional compactness property for convolutions with the Newtonian potential. Thus, \cite{CFRS} used here the linearity of the Laplace operator. Instead, in our quasilinear case, to get compactness we will rely on the following new weighted bounds \eqref{intro_IIder_p>2}-\eqref{intro_IIder_p<2} for the Hessian of a stable solution. These Hessian estimates are not present in \cite{CFRS}. In addition, proving them in the strong form of the following theorem, where $\norm{u}_{L^1(B_1)}$ appears on their right-hand side, will require an interpolation inequality  adapted to the $p$-Laplacian (Proposition \ref{prop5.2}) which is new, up to our knowledge, and is valid for all regular enough functions.

\begin{theorem}\label{thm_higher_int}
Let $u$ be a stable regular solution of $ -\Delta_p u=f(u) $ in $ B_1\subset\R^n $, with $ f\in C^1(\R)$ nonnegative. 
Then,
\begin{equation}
\label{intro_IIder_p>2}
\int_{B_{1/2}}\abs{\nabla u}^{p-2}\lvert D^2u\rvert\,dx\leq C\norm{u}_{L^1(B_1)}^{p-1} \qquad \text{if}\,\,\,p\geq2
\end{equation}
and
\begin{equation}
\label{intro_IIder_p<2}
\int_{B_{1/2}}\lvert D^2u\rvert\,dx\leq C\norm{u}_{L^1(B_1)} \qquad \text{if}\,\,\,p\in(1,2),
\end{equation}
where $C$ depends only on $n$ and $p$.
\end{theorem} 

In the proof of this result, as well as in other parts of the paper, we make use of the following important integrability results. For  $p>1$, Damascelli and Sciunzi~\cite{DS} proved that $\nabla u\in W^{1,2}_{\sigma,\text{loc}}(\Omega)$ when $u$ is a regular solution --- not necessarily stable --- to \eqref{plap_eq}; recall that $\sigma=\abs{\nabla u}^{p-2}$. In particular, one has
\begin{equation}\label{firstDS}
\int_{B_{1/2}}\abs{\nabla u}^{p-2}\lvert D^2u\rvert^2\,dx < +\infty.
\end{equation} 
At the same time, assuming $p\in(1,2)$, a classical result by Tolksdorf~\cite{T} ensures that $\nabla u\in W^{1,2}_{\text{loc}}(\Omega)$ for every regular solution $u$. Hence, both statements hold when $p\in(1,2)$.
Their results apply to general solutions, but do not provide quantitative bounds as the ones in \eqref{intro_IIder_p>2}-\eqref{intro_IIder_p<2}. However, we require $u$ to be stable and $f$ nonnegative. Note also the exponent~$1$ on $\lvert D^2u\rvert$ in \eqref{intro_IIder_p>2}-\eqref{intro_IIder_p<2}, while the exponent in \eqref{firstDS} is $2$.

\subsection{Application: regularity of extremal solutions}\label{subsec_extremal}
Let $\Omega\subset\R^n$ be a smooth bounded domain and assume that $f\in C^1([0,+\infty))$ satisfies \eqref{f_p_extremal_hp}. Given a constant $\lambda>0$, we consider the nonlinear elliptic problem involving the $p$-Laplacian

\stepcounter{equation}
\begin{equation*}\label{p_lambda_problem}
\left\{
\begin{array}{rcll}
-\plaplacian u&=& \lambda f(u) & \qquad \text{in}\,\,\Omega 
\\
u&>&0 & \qquad\text{in}\,\,\Omega
\\
u&=&0 & \qquad\text{on}\,\,\partial\Omega.
\\
\end{array}\right.
\eqno{(1.19)_{\lambda,p}}
\end{equation*}

In the following theorem we collect some known results concerning problem $(1.19)_{\lambda,p}$.

\begin{theorem}[\cite{CSa,CasSa,San1,San,CFRS}]\label{thm_extremal_known}
Let $\Omega\subset\R^n$ be a smooth bounded domain and $f\in C^1([0,+\infty))$ satisfy \eqref{f_p_extremal_hp}.

Then, there exists $ \lambda^*\in(0,+\infty) $ such that:
\begin{itemize}
	\item[(i)] For $ \lambda\in(0,\lambda^*)$, problem $(1.19)_{\lambda,p}$ admits a smallest regular solution $ u_\lambda $. The family $u_\lambda$ is increasing in $\lambda$ and every $u_\lambda$ is stable.
	\\
	For $ \lambda>\lambda^*$, problem $(1.19)_{\lambda,p}$ admits no regular solution.
	
	\item[(ii)] 
	$u^*:=\lim_{\lambda\uparrow\lambda^*}u_\lambda$
	is an energy solution of $(1.19)_{\lambda^*,p}$ --- in particular $u^*\in W^{1,p}(\Omega)$ --- under one of the following additional assumptions:
	\begin{itemize}
		\item[($\text{ii}_\text{a}$)] $p>2$, \eqref{f_p_convexity} holds, and either $\Omega$ is strictly convex or $n<p+\frac{p^2}{p-1}$; 
		\item[($\text{ii}_\text{b}$)] $p\in(1,2]$ and \eqref{f_p_convexity} holds.
	\end{itemize}
\end{itemize}
\end{theorem}

Statement (i) is proved in \cite[Theorem 1.4 (i)]{CSa} (among other references), while statement  $(\text{ii}_\text{a})$ for strictly convex $\Omega$ was proved by Castorina and the third author in \cite[Theorem~1.5]{CasSa}. The other statement in $(\text{ii}_\text{a})$ follows instead from \cite[Theorem 1]{San1}. For $p<2$, $(\text{ii}_\text{b})$ is a consequence of \cite[Theorem 2]{San} --- note here that the constant $\tau_-$ in \cite{San} satisfies $\tau_-\geq0$ since $f$ is convex. Finally, for $p=2$ $(\text{ii}_\text{b})$ is a consequence of the article \cite{CFRS} that we are extending.

Concerning the classical case $p=2$, problem $(1.19)_{\lambda,2}$ was first systematically studied for some model nonlinearities by Crandall and Rabinowitz \cite{CR}. 
For general nonlinearities $f\in C^1$ satisfying \eqref{f_p_extremal_hp} with $p=2$, it is a classical result that $u_\lambda$ converges in $L^1(\Omega)$ to $u^*$, which is a 
	weak\footnote{In the sense introduced by Brezis et al.\cite{BetA}: $u\in L^1(\Omega)$ is a weak solution of $(1.19)_{\lambda,2}$ if $f(u)\text{dist}(\cdot,\partial\Omega)\in L^1(\Omega)$ and 
	\[
	\int_\Omega\left(u\Delta \varphi+\lambda f(u)\varphi\right)\,dx=0
	\]
	for every $\varphi\in C^2(\overline\Omega)$ with $\varphi = 0$ on $\partial\Omega$.} 
solution of $(1.19)_{\lambda^*,2}$ and it is called the \textit{extremal solution} of problem~$(1.19)_{\lambda,2}$. 
We refer to~\cite{D} for a complete introduction to the extremal problem for the Laplacian.

In \cite{B,BV} Brezis and Vázquez raised several open questions about extremal solutions for problem $(1.19)_{\lambda,2}$. 
Concerning the regularity of $u^*$, they asked whether extremal solutions are energy solutions in every dimension and whether they are 
bounded in ``low'' dimensions, for all convex nonlinearities satisfying \eqref{f_p_extremal_hp} with $p=2$ and all domains.

While several partial answers were established throughout the last 25 years, both questions have been completely answered very recently in \cite{CFRS}. Indeed, assuming that $f$ is nonnegative, nondecreasing, convex, and superlinear at infinity, in \cite{CFRS} the authors proved
that, in every $C^3$ bounded domain $\Omega$, $u^*\in W^{1,2}_0(\Omega)$ for every dimension $n$ and, moreover, $u^*$ is smooth if $n\leq9$.
This result is optimal, since explicit examples of unbounded extremal solutions are well-known for $n\geq10$.

For the extremal problem with the $p$-Laplacian, a weak notion of solution in the sense of Brezis et al.~\cite{BetA} is not available. For this reason, for $p\neq2$ it is unknown apriori if~$u^*$ solves~$(1.19)_{\lambda^*,p}$ in some weak sense, unless one assumes $(\text{ii}_\text{a})$ or $(\text{ii}_\text{b})$ in Theorem~\ref{thm_extremal_known}.
However, when $f$ is the exponential or a power-type nonlinearity, for every $p > 1$ it is known that $u^*$ is an energy solution in every dimension --- see \cite{GP,GPP,CSa}. 

The following is our main result concerning the boundedness of the extremal function $u^*$ in strictly convex domains. It follows from combining Corollary \ref{cor_convex} with the $L^1(\Omega)$ bounds for $u_\lambda$, uniform in $\lambda$, from \cite{San,San1}. When $p\geq2$, it gives the optimal dimension $n$.

\begin{theorem}\label{thm_extremal_new}
	Let $n$ and $p$ satisfy \eqref{dimension}, $f\in C^1([0,+\infty))$ satisfy \eqref{f_p_extremal_hp} and \eqref{f_p_convexity}, and $\Omega$ be a smooth bounded domain which is strictly convex. Then,
	$u^*$ is bounded and it is therefore a regular solution.
\end{theorem}

We recall that the boundedness of the extremal solution was previously known, under assumptions \eqref{f_p_extremal_hp}-\eqref{f_p_convexity} on $f$, in general smooth domains for dimensions $n$ satisfying \eqref{intro-Ned-San}. For $p\geq2$ and $\Omega$ strictly convex, this was improved in \cite{CasSa} to the condition $n\leq p+2$, and even further in \cite{Miraglio} to condition \eqref{dimMiraglio} when $n\geq4$. Our Theorem \ref{thm_extremal_new} improves these results, reaching the optimal dimension when $p\geq2$ and $\Omega$ is strictly convex.

Instead, when $p\in(1,2)$ the best result on the boundedness of extremal solutions was the one by the third author~\cite{San}, which gave condition~\eqref{intro-Ned-San}. This is improved by our result $n<5p$ only in strictly convex domains and when $p\in(\frac{7}{4},2)$.

Under the same hypotheses of Theorem 1.5 but without assuming the bound \eqref{dimension} on the dimension, Castorina and the third author proved in \cite[Theorem 1.5]{CasSa} that $u^*\in W^{1,p}(\Omega)$ if $p\geq2$. This was also known in nonconvex domains for $p\in(1,2)$ and every dimension $n$ by \cite[Theorem 2]{San}, and for $p\geq2$ and some dimensions $n$ by \cite[Theorem 1]{San1}. 
These three papers showed that $u^*$ belongs to $W^{1,p}(\Omega)$ by establishing $L^q(\Omega)$ bounds for $u_\lambda$ and $f(u_\lambda)$, with $\lambda<\lambda^*$, that are uniform in $\lambda$. Then, the range of exponents $q\geq1$ allowed in their bounds lets them prove, through the equation $-\plaplacian u_\lambda = \lambda f(u_\lambda)$, a $W^{1,p}(\Omega)$ bound for $u_\lambda$ uniform in $\lambda$. Our results provide an alternative proof that, in every dimension, $u^*$ belongs to $W^{1,p}(\Omega)$ when $\Omega$ is a strictly convex domain. To prove this, we will use Corollary \ref{cor_convex} in conjunction with the $L^1(\Omega)$ bounds for $u_\lambda$, uniform in $\lambda$, from \cite{San, San1}. The details will be given within the proof of Theorem \ref{thm_extremal_new}.

\subsection{Morrey and $L^q$ estimates}\label{subsec_high_dim}
In light of the results from the previous sections, one can naturally wonder about the regularity of stable solutions to \eqref{plap_eq} in higher dimensions, that is, when condition \eqref{dimension} on $n$ and $p$ does not hold. When the dimension $n$ is strictly larger than the critical value in the right-hand side of \eqref{dimension}, we obtain almost optimal integrability estimates in Morrey spaces.

Instead, the critical case when there is equality in \eqref{dimension} seems to be delicate. The point here is to find an appropriate test function to which one applies the results of Section~\ref{sec_stability}. For instance, a similar test function as the one used for $p=2$ in \cite{CFRS} works only for $2\leq p<2+\sqrt{3}$. It is not clear to us which test function one could use for $p \geq 2+\sqrt{3}$.

We recall that, for $q\geq1$ and $\beta\in(0,n]$, the Morrey norms are defined as
\[
\norm{w}^q_{M^{q,\beta}(\Omega)}:=\sup_{\,x_0\in \overline{\Omega}, \,\rho>0}\,\rho^{\beta-n}\int_{\Omega\cap B_{\rho}(x_0)}\abs{w}^q\,dx.
\]
We state our result in the following theorem. A stronger version of it, showing also Morrey estimates for $\nabla u$, is given in Theorem \ref{thm_big_n_morrey}. Note that for the first statement in both results, we do not need $f$ to be nonnegative.

\begin{theorem}\label{thm_big_n}
	Let $u$ be a stable regular solution of 
	$-\Delta_pu=f(u)$ in $B_1\subset\R^n$, with $ f\in C^1(\R)$.
	Assume that
	\begin{equation}\label{big_dimension}
	\begin{cases}
	n> p+\frac{4p}{p-1}	\qquad &\text{for}\quad p\geq2, \\
	n> 5p		\qquad&\text{for}\quad p\in(1,2),
	\end{cases}
	\end{equation}
	and define
	\begin{equation*}
	q_n:=\begin{cases}
	\begin{aligned}
	& p+\frac{p^2}{n-2-2\sqrt{\frac{n-1}{p-1}}-p} 	&\text{if}\,\,\, n> p+\frac{4p}{p-1} \,\,\, \text{and}\,\,\, p\geq2,		\\
	& p+\frac{p^2}{n-2(p-1)-2\sqrt{(p-1)^2+n-p}-p} 	&\text{if}\,\,\, n> 5p \,\,\, \text{and}\,\,\, p\in(1,2).
	\end{aligned}
	\end{cases}
	\end{equation*}
	
	Then,
	\begin{equation}\label{estimate_big_n}
	\norm{u}_{M^{q,p+\frac{p^2}{q-p}}(B_{1/2})}
	\leq C\norm{\nabla u}_{L^p(B_{1})} \qquad\text{for every}\quad q\in(p,q_n),
	\end{equation} 
	where $ C$ depends only on $n$, $p$, and $q$. In addition, if $f$ is nonnegative then $\norm{\nabla u}_{L^p(B_{1})}$ can be replaced by $\norm{u}_{L^1(B_{1})}$ in the right-hand side of \eqref{estimate_big_n}.
\end{theorem}

For $p\geq2$, the result is almost optimal. Indeed, when
$n> p+\frac{4p}{p-1}$, in \cite{CSa,F} it is proved that the function $ u(x)=\abs{x}^{-p/(m-(p-1))}-1 $ is the extremal solution of
\begin{equation*}
\left\{
\begin{array}{rcll}
-\plaplacian u&=& \lambda^*(1+u)^m & \qquad \text{in}\,\,B_1 
\\
u&>&0 & \qquad\text{in}\,\,B_1
\\
u&=&0 & \qquad\text{on}\,\,\partial B_1,
\\
\end{array}\right.
\end{equation*}
with $m=\frac{n(p-1)-2\sqrt{(n-1)(p-1)}+2-p}{n-(p+2)-2\sqrt{(n-1)/(p-1)}}$. One can easily show that $ u\in M^{q,p+\frac{p^2}{q-p}}(B_{1/2}) $ if and only if $q\leq m+1=q_n$. 
The validity of \eqref{estimate_big_n} with $q=q_n$ for a general stable solution remains an open question, even for $p=2$.

\subsection{Remaining open questions} \label{subsec_open} 
In view of the results provided by~\cite{CFRS} for $p=2$, the following are the main problems which remain to be answered for general $p$. First, 
the extension of the boundary regularity for stable solutions developed in \cite{CFRS} --- a question that we have not attacked. Second, for $p\in(1,2)$, is $u^*$ bounded whenever $n<p+4p/(p-1)$? Regarding this question, proving boundedness in the interior of the domain would be already extremely interesting.
Third, as explained in the beginning of the previous Subsection \ref{subsec_high_dim}, the critical dimension in which there is equality in \eqref{dimension} remains largely open.

Finally, even for $p=2$ it is not clear if the nonnegativeness of the nonlinearity is needed for the boundedness results. For instance, from estimate \eqref{intro_stability} below with $\rho=1/2$ (which holds under no sign assumption on $f$), one can deduce BMO regularity for the stable solution. At the same time, the optimal boundedness result of \cite{CCaSa} in the radial case --- as well as several boundedness results in the nonradial setting and for nonoptimal dimension ranges \cite{C,C1,CasSa,Miraglio,N,San,San1} --- hold for every nonlinearity, independently of its sign.

\subsection{Main ingredients in the proofs}\label{subsec_ideas} Throughout the paper, we denote by
$$
r=r(x)=\abs{x} \quad\text{ and }\quad u_r=u_{r}(x)=\frac{x}{\abs{x}}\cdot\nabla u
$$ 
the modulus of $x$ and the radial derivative of $u$.

A key tool towards Theorem \ref{thm_interior} is the first choice\footnote{We will see here below that $\xi=\abs{\nabla u}\eta$, with $\eta$ a cut-off function, is a second important choice that we also need to use.} of the test function $\xi$ in the stability inequality \eqref{stability}. We take
\begin{equation}\label{first-test}
\xi:=\left(x\cdot\nabla u\right)r^{\frac{p-n}{2}}\phi(|\nabla u|/\varepsilon)\,\zeta,
\end{equation}
where $\phi$ and $\zeta$ are cut-off functions with $\phi(t)=0$ for $t\leq 1$, $\phi(t)=1$ for $t\geq2$, and $\zeta\in C_c^\infty(B_{3\rho/2})$ with $\zeta\equiv1$ in $B_\rho$, where $\rho<2/3$.
The presence of the cut-off $\phi$ is a first difference with the analysis in~\cite{CFRS} for $p=2$, where there was no need to use it. The reason to include it here is to avoid the set $\{ \nabla u=0\}$, where the ellipticity of the equation degenerates. Indeed, we will need to compute two weak derivatives of $\xi$, that is, three weak derivatives of $u$, which we know to exist in the regular set $\{ |\nabla u|>0\}$. In addition, for errors to be small, we will need the important regularity result \eqref{firstDS} of \cite{DS}, stating that $\nabla u \in W^{1,2}_{\sigma,{\rm loc}}(B_1)$.

Plugging $\xi$ in \eqref{stability} --- see the series of lemmata in Section \ref{sec_stability} --- we will obtain
\begin{equation}\label{stab_intro}
\begin{split}
&\hspace{-0.5cm}\left(p(n-p)-\left(\frac{n-p}{2}\right)^2\right)\int_{B_\rho}r^{p-n}\abs{\nabla u}^{p-2}u_{r}^2\,dx
\\
&\hspace{1cm}-(p-2)\left(\frac{n-p}{2}\right)^2\int_{B_\rho}r^{p-n}\abs{\nabla u}^{p-4}u_{r}^4\,dx \leq C\int_{B_{3\rho/2}\setminus B_\rho}r^{p-n}\abs{\nabla u}^p\,dx.
\end{split}
\end{equation}
The constant in front of the first integral is positive whenever $p<n<5p$.
When $p\geq2$, this condition on $n$ is not restrictive, since in this case $p+4p/(p-1)\leq5p$. 
Also for $p\geq2$, the second term in the left-hand side of \eqref{stab_intro} is nonpositive and we can merge it with the first one by using $u_{r}^2\leq\abs{\nabla u}^2$.
In this way, we obtain a constant in front of the resulting integral in the left-hand side which is positive whenever $p<n<p+4p/(p-1)$. Hence, under this assumption on $n$ and if $p\geq2$, from \eqref{stab_intro} we obtain
\begin{equation}\label{intro_stability}
\int_{B_\rho}r^{p-n}\abs{\nabla u}^{p-2}u_r^2\,dx \leq C \rho^{p-n}\int_{B_{3\rho/2}\setminus B_\rho}\abs{\nabla u}^p\,dx.
\end{equation}
However, the requirement $p<n$ will not be needed in our results, thanks to the following:
\begin{remark}\label{artificial}
When $n\leq p$, we can add additional artificial variables and consider
the solution $u$ to be defined in a ball of a greater dimensional Euclidean space, of dimension larger than $p$.
The key point here is that this procedure preserves the stability condition~\eqref{stability}. For this, one uses Fubini's theorem and that $|\nabla_{x}\xi |^2 \leq |\nabla_{\bar{x}} \xi |^2$, where $x\in\R^n$ while $\bar{x}$ denotes the variables in the larger dimensional
space. However, since in \eqref{stability} $p-2$ could be negative, it is important to notice that $\nabla_x u\cdot\nabla_x\xi=\nabla_{\bar{x}}u\cdot\nabla_{\bar{x}}\xi$.
This will be the case --- even though we will surely have $\nabla_x\xi\neq\nabla_{\bar{x}}\xi$ --- since $u$ only depends on the $x$-variables. Finally, note that this procedure works since our interior estimates do not rely on any boundary datum.	
\end{remark}

In contrast, when $p\in(1,2)$, we have that $5p<p+4p/(p-1)$, and in order to have a positive constant in front of the first integral in \eqref{stab_intro} we need to assume the more restrictive assumption $n<5p$. Notice here that when $p\in(1,2)$ we cannot use $u_{r}^2\leq\abs{\nabla u}^2$ on the second term in \eqref{stab_intro} in order to merge it with the first one, since this second term is positive. In conclusion, when $p\in(1,2)$ we simply neglect the second term in the left-hand side of \eqref{stab_intro}, losing information.

At this point, it is important to point out that there exist unbounded energy solutions to~\eqref{plap_eq} which satisfy inequality~\eqref{intro_stability} for all $\rho$ small enough. An explicit example is provided in the case $p=2$, $n=3$ in \cite[Remark 2.2]{CFRS}. This means that we still need to use stability in a crucial way in order to prove our results.

To do this, we follow a key idea in \cite{CFRS}, which consists of showing that the weighted\footnote{At scale $\rho=1$, the weight is $\abs{\nabla u}^{p-2}$, as in the left-hand side of \eqref{intro_stability}. In the classical case $p=2$, such weight did not appear, obviously.} $L^2$-norm of the radial derivative $u_r$ is comparable to the $L^p$-norm of the full gradient at every scale, at least assuming a doubling assumption on the $L^p$ norm of the gradient. This is the content of Lemma~\ref{lemma_doubling}. The proof of this fact is based on a delicate compactness argument, which relies on new apriori estimates that we explain below.
Once Lemma \ref{lemma_doubling} is available, from \eqref{intro_stability} we deduce that
\[
\int_{B_\rho}r^{p-n}\abs{\nabla u}^{p-2}u_r^2\,dx
\leq
C\int_{B_{3\rho/2}\setminus B_\rho}r^{p-n}\abs{\nabla u}^{p-2}u_r^2\,dx
\]
and, from this, with a hole-filling and iteration argument we conclude that $u$ is $C^\alpha$ in the interior.

The compactness argument in the proof of Lemma \ref{lemma_doubling} turns out to be more delicate than for the case $p=2$ in \cite{CFRS}. Indeed, the Newtonian potential --- and hence the linearity of the Laplace operator --- was used in \cite{CFRS} in an important way to get compactness. Instead, our proof will rely on the following three new estimates for a stable regular solution:
\begin{equation}\label{intro_Lp_controls_IIder_p>2}
\int_{B_{1/2}}\abs{\nabla u}^{p-2}\lvert D^2u\rvert\,dx\leq C\norm{\nabla u}_{L^p(B_1)}^{p-1} \qquad \text{if}\,\,\,p\geq2,
\end{equation}
\begin{equation}\label{intro_Lp_controls_IIder_p<2}
\int_{B_{1/2}}\lvert D^2u\rvert\,dx\leq C\norm{\nabla u}_{L^p(B_1)} \qquad \text{if}\,\,\,p\in(1,2),
\end{equation}
and
\begin{equation}\label{intro_Lp_controls_p+gamma}
\norm{\nabla u}_{L^{p+\gamma}(B_{1/2})} \leq C \norm{\nabla u}_{L^p(B_1)},
\end{equation}
where $\gamma$ and $C$ are positive constants depending only on $n$ and $p$. Note that these estimates are a first step towards the ones in our main results --- \eqref{intro_IIder_p>2}, \eqref{intro_IIder_p<2}, and \eqref{L1controls_p}, respectively --- which have $\norm{u}_{L^1(B_1)}$ in their right-hand sides instead of $\norm{\nabla u}_{L^p(B_1)}$.

A fundamental tool in obtaining estimates \eqref{intro_Lp_controls_IIder_p>2}, \eqref{intro_Lp_controls_IIder_p<2}, and \eqref{intro_Lp_controls_p+gamma} is a geometric form of the stability condition that was proved by Farina, Sciunzi, and Valdinoci \cite{FSV,FSV1} and is reported in the following theorem. It extends to the $p$-Laplacian the well-known inequality of Sternberg and Zumbrun for stable solutions to semilinear equation.

\begin{theorem}[Farina, Sciunzi, and Valdinoci \cite{FSV, FSV1}]\label{thm_SZ_stability}
	Let $ p\in(1,+\infty)$ and $f\in C^1(\R)$. For every stable regular solution $ u $ of $ -\Delta_p u=f(u) $ in $ \Omega\subset\R^n $ and for every $\eta\in C^1_c(\Omega)$, it holds that \footnote{In the previous and printed version of this article, the integral in the right-hand side of \eqref{SZ_formula} was $(p-1)\int_{\Omega}\abs{\nabla u}^{p} \abs{\nabla \eta}^2 \,dx$ and thus incorrect ---the reason being that for $p=1$, it cannot vanish identically. The correct one, above, is taken from \cite[Theorem 2.5]{FSV1}. However, this correction does not affect the validity of all our arguments in the previous and printed versions of the paper, since we only used \eqref{SZ_formula} to bound its right-hand side by the quantity $C\int_{\Omega}\abs{\nabla u}^{p} \abs{\nabla \eta}^2 \,dx$ with $C$ a constant depending only on $n$ and $p$. This is a correct bound that still holds independently of the error.}
	\begin{equation}\label{SZ_formula}
	\begin{split}
	&(p-1)\int_{\Omega\cap\{|\nabla u|>0\}}\abs{\nabla u}^{p-2}\abs{\nabla_T\abs{\nabla u}}^2\eta^2\,dx+\int_{\Omega\cap\{|\nabla u|>0\}}\abs{A}^2\abs{\nabla u}^p \eta^2\,dx\\ 
	&\hspace{2cm} 	
	\leq \int_{\Omega}\abs{\nabla u}^{p-2}\left\{ |\nabla u|^2 \abs{\nabla \eta}^2+(p-2) (\nabla u \cdot\nabla \eta)^2\right\}    \,dx.
	\end{split} 
	\end{equation}
\end{theorem}
Here $\nabla_T$ is the tangential gradient to the level set of $u$ passing through a given point $x$ and 
$\abs{A}^2:=\sum_{i=1}^{n-1}\kappa^2_i$ is the second fundamental form of the level set; $\kappa_i$ are the $n-1$ principal curvatures of the level set of $u$. 
These objects are well defined in $\{|\nabla u|>0\}$, since here the equation $-\plaplacian u=f(u)$ is uniformly elliptic and therefore every regular solution is of class $C^2$ in this set --- see Remark~\ref{rmk_integrability} below. 
In particular, for $x\in \Omega\cap\{\abs{\nabla u}>0\}$, the level set of $u$ passing through $x$ is a $C^2$ embedded hypersurface of $\R^n$.

The proof of Theorem \ref{thm_SZ_stability} relies on the choice of test function 
\[\xi=\abs{\nabla u}\eta\]
in the stability inequality \eqref{stability}, where $\eta$ is the $C^1_c(\Omega)$ function appearing in \eqref{SZ_formula}.

Using Theorem \ref{thm_SZ_stability} we establish the interior estimates \eqref{intro_Lp_controls_IIder_p>2}, \eqref{intro_Lp_controls_IIder_p<2}, and \eqref{intro_Lp_controls_p+gamma} above. 
The proofs of~\eqref{intro_Lp_controls_IIder_p>2}-\eqref{intro_Lp_controls_IIder_p<2}, contained in Lemma~\ref{lemma_II_der}, and the one of~\eqref{intro_Lp_controls_p+gamma}, contained in Lemma~\ref{lemma_Lpgamma}, are independent and do not rely on each other. 
Moreover, we prove the higher integrability $W^{1,p+\gamma}$ estimate in two different ways. The first proof, presented in Section~\ref{sec_higher}, is based on a similar argument to~\cite[Proposition 2.4]{CFRS}, controlling first the $L^1$-norm of the $(p+1)$-Laplacian of~$u$ --- this will be estimate \eqref{pplus1} in Theorem \ref{thm_interior} still with $\norm{u}_{L^1(B_1)}$ replaced by $\norm{\nabla u}_{L^p(B_1)}$ on its right-hand side. The second proof of the higher integrability, presented in Appendix \ref{appendix_alt_proof}, relies instead on the Sobolev inequality of Michael-Simon and Allard and we present it simply as an interesting alternative method to obtain this type of estimates. Both proofs give explicit values of $\gamma$, depending only on $n$ and $p$, that we collect in Remark \ref{rmk_gamma}. However, for the scope of the present paper, we only need \eqref{L1controls_p} for some $\gamma>0$.

Improving our H\"older estimate, as well as the bounds \eqref{intro_Lp_controls_IIder_p>2}-\eqref{intro_Lp_controls_IIder_p<2}-\eqref{intro_Lp_controls_p+gamma}, by replacing $\norm{\nabla u}_{L^p(B_1)}$ by $\norm{u}_{L^1(B_1)}$ turns out to be more delicate than in the linear case $p=2$ --- but still based on an interpolation inequality. For this, we establish a new interpolation result adapted to the $p$-Laplacian, Proposition \ref{prop5.2}, involving the $L^1$-norm of $D^2u$ weighted by $\abs{\nabla u}^{p-2}$. It holds for $p\geq2$ and all regular enough functions $u$.
For $p\in(1,2)$ instead, we will use a classical interpolation inequality, as stated, e.g., in \cite[Theorem 7.28]{GT}.

Finally, the proof of Theorem \ref{thm_big_n} in higher dimensions is obtained by choosing $\xi=(x\cdot\nabla u)r^{-a/2}\phi(|\nabla u|/\varepsilon)\zeta$ as test function in the stability inequality \eqref{stability}, for appropriate exponents $a>0$.

\subsection{Structure of the paper} In Section~\ref{sec_stability} we use stability with a specific test function in order to prove the key Lemma~\ref{lemma_stability}. In Section~\ref{sec_higher} we prove higher integrability estimates in the interior, i.e., \eqref{intro_Lp_controls_IIder_p>2}, \eqref{intro_Lp_controls_IIder_p<2}, and \eqref{intro_Lp_controls_p+gamma}. In Section~\ref{sec_doubling} we establish the control of the full gradient by the weighted radial derivative, Lemma~\ref{lemma_doubling}. In Section~\ref{sec_interior} we prove Theorem~\ref{thm_interior}, our global result in convex domains (Corollary~\ref{cor_convex}), and Theorem~\ref{thm_extremal_new} on extremal solutions. Finally, Section~\ref{sec_big_n} deals with the case of higher dimensions, proving Theorem~\ref{thm_big_n}.
We collect in the appendices some technical lemmata and an alternative proof of the higher integrability estimate of Lemma \ref{lemma_Lpgamma} which uses the Michael-Simon and Allard inequality.

\section{The key lemma}\label{sec_stability}
This section is devoted to the proof of the following lemma, an interior estimate for stable solutions which holds for every $C^1$ nonlinearity, not depending on its sign. 
We establish it using the stability condition \eqref{stability} with test function 
\[\xi=(x\cdot\nabla u)\,r^{(p-n)/2}\phi(\abs{\nabla u}/\e)\zeta,\]
where $r=\abs{x}$, $0\leq\phi\leq1$, $\phi(t)=0$ if $t\leq1$, $\phi(t)=1$ if $t\geq2$, and $\zeta\in C_c^\infty(B_1)$ is a cut-off function. 

For every $ C^1 $ function $ \varphi $ we use the following notation for its radial derivative:
\[
\varphi_r=\varphi_r(x)=\frac{x}{\abs{x}}\cdot\nabla \varphi(x).
\]

\begin{lemma}\label{lemma_stability}
	Assume that $n>p$, and that $n$ and $p$ satisfy \eqref{dimension}. Let $u$ be a stable regular solution of $-\Delta_p u=f(u)$ in $B_1\subset\R^n$, with $f\in C^1(\R)$.
	
	Then, for every $\rho\in(0,2/3) $ we have
	\begin{equation}\label{stability_estimate}
	\int_{B_\rho}r^{p-n}\abs{\nabla u}^{p-2}u_r^2\,dx \leq C \rho^{p-n}\int_{B_{3\rho/2}\setminus B_\rho}\abs{\nabla u}^p\,dx,
	\end{equation}
	where $C$ is a constant depending only on $n$ and $p$.
\end{lemma}

Before starting the proof of Lemma \ref{lemma_stability}, we need to recall some classical regularity results for nonlinear equations involving the $p$-Laplacian. This will allow us to use distributional second and third derivatives of a regular solution $u$ to $-\Delta_p u=f(u)$.

\begin{remark}\label{rmk_integrability}
	By classical results \cite{DiB,T}, every regular solution $u$ to \eqref{plap_eq} in $B_1\subset\R^n$ is~$C^{1,\vartheta}$ for some $\vartheta>0$. Furthermore, the solution $u$ is $C^2$ in $B_1\cap\{\abs{\nabla u}>0\}$, since it solves a uniformly elliptic equation in the neighborhood of every point $x\in \{\abs{\nabla u}>0\}$ --- see for instance \cite[Corollary~2.2]{DS}. To show this, one breaks the divergence in \eqref{plap_eq} to obtain an equation in nondivergence form with H\"older coefficients, for which one can use Schauder theory.
	
	At the same time, since $f(u)\in C^1(B_1)$, we can differentiate the equation in $B_1\cap\{\abs{\nabla u}>0\}$ and then see it as a nondivergence form equation for $u_{x_i}$. Then, by Calderon-Zygmund theory we deduce that $ u\in W^{3,q}_\text{loc}(B_1\cap\{\abs{\nabla u}>0\}) $ for every $q<\infty$ --- see \cite[Chapter 9]{GT}.
	
	With these regularity results at hand, we are allowed to work with weak second and third derivatives of $u$ throughout the following lemmata, since we will always work in the set $\{\abs{\nabla u}>0\}$.

	Still for a regular solution $u$ to \eqref{plap_eq} in $B_1\subset\R^n$, \cite[Theorem 2.2]{DS} established that, for $p>1$ and $\rho<1$,
	\begin{equation}\label{star}
	\int_{B_\rho\cap\{\abs{\nabla u}>0\}}\abs{\nabla u}^{p-2}\lvert D^2u\rvert^2\,dx < +\infty
	\end{equation}
	(take $\beta=\gamma=0$ in their theorem). With this estimate at hand and a regularizing procedure described in \cite[Corollary 2.2]{DS}, in that paper it was proved that $\abs{\nabla u}^{p-2}\nabla u\in W^{1,2}_\textnormal{loc}(B_1)$ and $\nabla u\in W^{1,2}_{\sigma,\textnormal{loc}}(B_1)$ with $\sigma=\abs{\nabla u}^{p-2}$. In particular, $\abs{\nabla u}^{p-1}\in W^{1,2}_{\textnormal{loc}}(B_1)$. From this and Stampacchia's theorem --- see \cite[Theorem 6.19]{LL} --- the integrands in \eqref{star} and in \eqref{intro_IIder_p>2} are equal to zero a.e. in the critical set $\{\nabla u=0 \}$. Thus, both integrals agree when computing them over $B_\rho$ or $B_\rho\cap\{\abs{\nabla u
	}>0\}$.

	If $p\in(1,2)$, a classical result from \cite{T} ensures that $ \nabla u\in W^{1,2}_\text{loc}(B_1) $. As before, from this we deduce that the integral $\int_{B_{1/2}}\lvert D^2u\rvert\,dx$ in \eqref{intro_IIder_p<2} agrees with $\int_{B_{1/2}\cap\{\abs{\nabla u}>0\}}\lvert D^2u\rvert\,dx$ when $p\in(1,2)$. Both the contributions from \cite{DS} and \cite{T}, that we use several times throughout the paper, are extensively used also in~\cite{FSV,FSV1,CasSa}.
	
	These results from \cite{DS} and \cite{T} will ensure that all the integrals in our upcoming computations for the first term in the stability condition \eqref{stability} are well defined and can be computed either in $B_1$ or in $B_1\cap\{\abs{\nabla u}>0\}$. This will be a consequence of our two choices for $\xi$, which will always include $\nabla u$ as a factor on them.
	
	More generally, for $p\neq2$ and a general test function $\xi\in\mathcal{T}_u$ in \eqref{stability} we also have that the first integral in \eqref{stability} agrees with its value when computed in $\{\abs{\nabla u}>0\}$. This is obvious when $p>2$ since $\abs{\nabla u}^{p-2}=0$ in the critical set. Instead, for $p\in(1,2)$ it is the definition of~$\mathcal{T}_u$ which forces $\nabla\xi=0$ a.e. in $\{\abs{\nabla u}=0\}$, since $\abs{\nabla u}^{p-2}=+\infty$ a.e. in the critical set.
	
	Finally, even if we will not use this fact, it is worth mentioning Lou's result \cite{Lou} stating that the critical set $\{\nabla u=0\}$ has null measure if $f$ is positive and $u$ is a regular solution of~\eqref{plap_eq}, not necessarily stable.
\end{remark}

We split the proof of Lemma \ref{lemma_stability} into five lemmata. To state the first one, we consider
\[
\boldsymbol{c}(x):=x\cdot\nabla u(x)
\]
and
\[
\phi_\e:=\phi\left(\frac{\abs{\nabla u}}{\e}\right),
\]
where $\phi\in C^\infty(\R)$ takes values into $[0,1]$, $\phi(t)=1$ if $t\geq2$, and $\phi(t)=0$ if $t\leq1$.
 
\begin{lemma}\label{lemma_stab1}
	Let $u$ be a stable regular solution of $-\plaplacian u =f(u)$ in $B_1\subset\R^n$, with $f\in C^1(\R)$, $\phi_\e$ as defined above, $\eta$ a Lipschitz function with compact support in $B_1$, and $\boldsymbol{c}=x\cdot \nabla u$. Then,
	\begin{equation}\label{stab1}
	\begin{split}
	&\hspace{-0.2cm}\int_{B_1}\boldsymbol{c}\Big\{\textnormal{div}\left(\abs{\nabla u}^{p-2}\nabla \boldsymbol{c}\right)+f'(u)\boldsymbol{c}+(p-2)\textnormal{div}\left(\abs{\nabla u}^{p-4}(\nabla \boldsymbol{c}\cdot\nabla u)\nabla u\right)\Big\}\eta^2\phi_\e^2\,dx 
	\\
	&\hspace{1cm}\leq
	\int_{B_1}\boldsymbol{c}^2\Big\{\abs{\nabla u}^{p-2}\abs{\nabla\eta}^2+(p-2)\abs{\nabla u}^{p-4}(\nabla\eta\cdot\nabla u)^2\Big\}\phi_\e^2\,dx+o(1),
	\end{split}
	\end{equation}
	where $o(1)$ indicates a quantity that goes to zero as $\e\downarrow0$, the derivatives and divergences in the previous integrals are understood in the weak sense, and the integrands are $L^1(B_{1})$ functions.
\end{lemma}

\begin{proof}	
	First,
	by Remark \ref{rmk_integrability} the weak second derivatives of $\boldsymbol{c}(x)=x\cdot\nabla u(x)$ are well defined in $\{\abs{\nabla u}>\e\}$.

	We plug  
	$\xi=\xi_\e=\boldsymbol{c}\,\eta\,\phi_\e$ in the stability inequality \eqref{stability}. 
	It is clear that $\xi_\e$ belongs to the space of test functions $\mathcal{T}_u$, since the support of $\xi_\e$ is contained in $B_\rho\cap\{\abs{\nabla u}\geq\e\}$ for some $\rho<1$ and, in this set, $\abs{\nabla u}^{p-2}$ is bounded below and above by positive constants (even when $p<2$) and at the same time $u$ is $C^2$ --- hence $D^2u$ is bounded in this set. 
	
	Next, since $\nabla u\in W^{1,2}_{\sigma,\text{loc}}(B_1) $ --- see Remark \ref{rmk_integrability} --- we have that
	\begin{equation}\label{funct_space1}
	\int_{B_1}\abs{\nabla u}^{p-2}\lvert D^2u\rvert^2\eta^2\,dx <+\infty,
	\end{equation}
	and combining \eqref{funct_space1} with a Cauchy-Schwarz inequality we also deduce that
	\begin{equation}\label{funct_space2}
	\int_{B_1}\abs{\nabla u}^{p-1}\lvert D^2u\rvert\eta\,dx <+\infty.
	\end{equation}
	These two facts will be crucial to let $\e\downarrow0$ later in the proof, through the use of the dominated convergence theorem.
	
	Now, we plug $\xi_\e\in\mathcal{T}_u$ in the stability inequality \eqref{stability}, obtaining for the first term
	\begin{equation}\label{cutoff_1}
	\begin{split}
	&\int_{B_1}\abs{\nabla u}^{p-2}\abs{\nabla\xi_\e}^2\,dx
	\\
	&\hspace{0.5cm}=\int_{B_1} \left\{\boldsymbol{c}^2\abs{\nabla u}^{p-2}\abs{\nabla\eta}^2+\eta^2\abs{\nabla u}^{p-2}\abs{\nabla \boldsymbol{c}}^2+\boldsymbol{c}\abs{\nabla u}^{p-2}\nabla \boldsymbol{c}\cdot\nabla\eta^2\right\}\phi^2_\e\,dx
	\\
	&\hspace{1cm}+\int_{B_1}\abs{\nabla u}^{p-2}\eta^2\boldsymbol{c}^2\abs{\nabla\phi_\e}^2\,dx
	+2\int_{B_1}\abs{\nabla u}^{p-2}\boldsymbol{c}\,\eta\,\phi_\e\,\nabla\phi_\e\cdot\left(\eta\nabla\boldsymbol{c}+\boldsymbol{c}\nabla\eta\right)\,dx.
	\end{split}
	\end{equation}
	We claim that the terms in \eqref{cutoff_1} in which $\nabla\phi_\e$ appears, i.e., the second and the third one in the right-hand side, go to zero as $\e\downarrow0$. 
	In order to show this, and also to control by $o(1)$ other terms that will appear later, we first note that
	\begin{equation}\label{cutoff_2}
	\abs{\nabla\phi_\e}
	=\abs{\frac{1}{\e}\phi'\left(\frac{\abs{\nabla u}}{\e}\right)\nabla\abs{\nabla u}}
	=\abs{\frac{1}{\e}\phi'\left(\frac{\abs{\nabla u}}{\e}\right)D^2u\cdot\frac{\nabla u}{\abs{\nabla u}}}
	\leq\frac{1}{\e}\abs{\phi'\left(\frac{\abs{\nabla u}}{\e}\right)}\lvert D^2u\rvert.
	\end{equation}	
	We use now that $\abs{\phi'(\abs{\nabla u}/\e)}$ is bounded independently of $\e$ and it is supported in $\{\e<|\nabla u|<2\e\}$, that in this set we have $\boldsymbol{c}^2\leq\abs{\nabla u}^2\leq 4\e^2$, and also that $\nabla\boldsymbol{c}=\nabla u+D^2u\cdot x$, to obtain
	\begin{equation}\label{cutoff_3}
	\begin{split}
	&\hspace{-0.2cm}\int_{B_1}\abs{\nabla u}^{p-2}\eta^2\boldsymbol{c}^2\abs{\nabla\phi_\e}^2\,dx
	+ \int_{B_1}\abs{\nabla u}^{p-2}\eta^2\abs{\boldsymbol{c}\,\nabla\boldsymbol{c}}\abs{\phi_\e\nabla\phi_\e}\,dx
	\\
	&\hspace{7cm}
	+ \int_{B_1}\abs{\nabla u}^{p-2}\abs{\eta\nabla\eta}\boldsymbol{c}^2\abs{\phi_\e\nabla\phi_\e}\,dx
	\\
	&\hspace{0.5cm}\leq C\int_{\{\e<\abs{\nabla u}<2\e\}}\abs{\nabla u}^{p-2}\Big\{\eta^2\lvert D^2u\rvert^2+\abs{\eta}\abs{\nabla u}\lvert D^2u\rvert	\Big\}\,dx
	= o(1) \qquad \text{as}\,\,\e\downarrow0,
	\end{split}
	\end{equation}
	where $C$ is a constant independent of $\e$. That the last integral tends to $0$ as $\e\downarrow0$ follows from \eqref{funct_space1}, \eqref{funct_space2}, and dominated convergence.
	
	Therefore, using \eqref{cutoff_3}, we deduce from \eqref{cutoff_1} that
	\begin{equation}\label{sstab1}
	\begin{split}
	&\hspace{-0.5cm}\int_{B_1}\abs{\nabla u}^{p-2}\abs{\nabla\xi_\e}^2\,dx=\int_{B_1}\phi^2_\e\boldsymbol{c}^2\abs{\nabla u}^{p-2}\abs{\nabla\eta}^2\,dx
	\\
	&\hspace{0.8cm}+\int_{B_1}\phi^2_\e\eta^2\abs{\nabla u}^{p-2}\abs{\nabla \boldsymbol{c}}^2\,dx
	+\int_{B_1}\phi^2_\e\boldsymbol{c}\abs{\nabla u}^{p-2}\nabla \boldsymbol{c}\cdot\nabla\eta^2\,dx+o(1).
	\end{split}
	\end{equation}
	We now integrate by parts the third term in the right-hand side to get
	\begin{equation}\label{sstab2}
	\begin{split}
	&\int_{B_1}\phi^2_\e\boldsymbol{c}\abs{\nabla u}^{p-2}\nabla \boldsymbol{c}\cdot\nabla\eta^2\,dx 
	=
	-\int_{B_1}\phi^2_\e\eta^2\textnormal{div}(\boldsymbol{c}\abs{\nabla u}^{p-2}\nabla \boldsymbol{c})\,dx
	\\
	&\hspace{3cm}
	-\int_{B_1}\eta^2\boldsymbol{c}\abs{\nabla u}^{p-2}\nabla \boldsymbol{c}\cdot\nabla\phi^2_\e\,dx
	\\
	&\hspace{1cm}=-\int_{B_1}\phi^2_\e\eta^2\abs{\nabla u}^{p-2}\abs{\nabla \boldsymbol{c}}^2\,dx
	-\int_{B_1}\phi^2_\e\eta^2\boldsymbol{c}\,\textnormal{div}(\abs{\nabla u}^{p-2}\nabla \boldsymbol{c})\,dx
	+o(1),
	\end{split}
	\end{equation}
	where we used \eqref{cutoff_3} to control the term containing $\nabla\phi_\e$ with $o(1)$. Thus, plugging~\eqref{sstab2} in~\eqref{sstab1}, we exploit a simplification and obtain
	\begin{equation}\label{sstab3}
	\begin{split}
	&\int_{B_1}\abs{\nabla u}^{p-2}\abs{\nabla\xi_\e}^2\,dx
	\\
	&\hspace{0.8cm}=\int_{B_1}\phi^2_\e\boldsymbol{c}^2\abs{\nabla u}^{p-2}\abs{\nabla\eta}^2\,dx-\int_{B_1}\phi^2_\e\eta^2\boldsymbol{c}\,\textnormal{div}(\abs{\nabla u}^{p-2}\nabla \boldsymbol{c})\,dx
	+o(1).
	\end{split}
	\end{equation}
	
	Computing now the second term in the stability inequality \eqref{stability} for the test function $ \xi_\e$, we find
	\begin{equation*}
	\begin{split}
	&\hspace{-5mm}\int_{B_1}(p-2)\abs{\nabla u}^{p-4}\left(\nabla u\cdot\nabla\xi_\e\right)^2dx
	\\
	&=
	\int_{B_1} (p-2)\phi^2_\e\boldsymbol{c}^2\abs{\nabla u}^{p-4}\left(\nabla u\cdot\nabla\eta\right)^2dx
	+
	\int_{B_1} (p-2)\phi^2_\e\eta^2\abs{\nabla u}^{p-4}\left(\nabla\boldsymbol{c}\cdot\nabla u\right)^2dx
	\\
	&\hspace{3cm}+\int_{B_1} (p-2)\phi^2_\e\abs{\nabla u}^{p-4}\boldsymbol{c}\left(\nabla u\cdot\nabla\eta^2\right)\left(\nabla\boldsymbol{c}\cdot\nabla u\right)\,dx+o(1),
	\end{split}
	\end{equation*}
	where all the terms containing $\nabla\phi_\e$ have been controlled within $o(1)$ using~\eqref{cutoff_3}.
	Integrating by parts the second term in the right-hand side, exploiting a cancellation, and seeing again that the terms containing $\nabla\phi_\e$ tend to zero as $\e\downarrow0$ thanks to \eqref{cutoff_3}, we get
	\begin{equation}\label{sstab4}
	\begin{split}
	&\int_{B_1}(p-2)\abs{\nabla u}^{p-4}\left(\nabla u\cdot\nabla\xi_\e\right)^2dx=
	\int_{B_1} (p-2)\phi^2_\e\boldsymbol{c}^2\abs{\nabla u}^{p-4}\left(\nabla u\cdot\nabla\eta\right)^2dx
	\\
	&\hspace{3cm}-\int_{B_1}(p-2)\phi^2_\e\eta^2\boldsymbol{c}\,\textnormal{div}\left(\abs{\nabla u}^{p-4}(\nabla\boldsymbol{c}\cdot\nabla u)\nabla u \right)\,dx+o(1).
	\end{split}
	\end{equation}
	Finally, using \eqref{sstab3} and \eqref{sstab4} in the stability inequality \eqref{stability} with $\xi=\xi_\e$, we conclude the lemma.
\end{proof}

Next step towards the proof of Lemma \ref{lemma_stability} consists of proving an identity that involves the weak third derivatives of $u$. It holds in the distributional sense in $B_1\cap\{\abs{\nabla u}>0 \}$ and does not require the stability of the regular solution. Recall that $ u\in W^{3,q}_\text{loc}(B_1\cap\{\abs{\nabla u}>0\}) $ for every $q<\infty$ --- see Remark \ref{rmk_integrability}.
After proving this identity, we will plug it in \eqref{stab1}, exploit a simplification in the third derivatives of $u$, and send $\e\downarrow0$. This limit argument will be carried out in Lemma~\ref{lemma_stab3}.

\begin{lemma}
	Let $u$ be a regular solution of $-\plaplacian u =f(u)$ in $B_1\subset\R^n$, with $f\in C^1(\R)$, and $\boldsymbol{c}=x\cdot \nabla u$. Then, the identity 
	\begin{equation}\label{stab2}
		\begin{aligned}
		\textnormal{div}\left(\abs{\nabla u}^{p-2}\nabla \boldsymbol{c}\right)+f'&(u)\boldsymbol{c}+(p-2)\,\textnormal{div}\left(\abs{\nabla u}^{p-4}(\nabla \boldsymbol{c}\cdot \nabla u)\nabla u\right)
		\\
		&\hspace{2cm}=p\,\textnormal{div}\left(\abs{\nabla u}^{p-2}\nabla u\right)
		\end{aligned}
	\end{equation}
	holds in the weak sense in $B_1\cap\{\abs{\nabla u}>0\}$.
	\begin{proof}
		Throughout the proof we will use the implicit summation over repeated indexes and we will denote the derivatives with subscripts --- $u_j$ denoting the derivative $u_{x_j}$.
		
		We start the proof by claiming the following identity:
		\begin{equation}\label{claim2}
		\begin{aligned}
		\textnormal{div}\left(\abs{\nabla u}^{p-2}\nabla \boldsymbol{c}\right)&+f'(u)\boldsymbol{c}
		\\
		&=2\,\textnormal{div}\left(\abs{\nabla u}^{p-2}\nabla u\right) 
		-\textnormal{div}\big((x\cdot\nabla\abs{\nabla u}^{p-2})\nabla u\big).
		\end{aligned}
		\end{equation}
		To prove it, we recall that $\nabla \boldsymbol{c}=\nabla u+D^2u\cdot x$ and we write the first term in \eqref{claim2} as
		\begin{align*}
			\textnormal{div}\left(\abs{\nabla u}^{p-2}\nabla \boldsymbol{c}\right)&=\textnormal{div}\left(\abs{\nabla u}^{p-2}\nabla u\right)+\left(\abs{\nabla u}^{p-2}u_{ij}x_j\right)_i \\
			&=\textnormal{div}\left(\abs{\nabla u}^{p-2}\nabla u\right)+\abs{\nabla u}^{p-2}\Delta u+\left(\abs{\nabla u}^{p-2}u_{ij}\right)_ix_j
			\\
			&=2\,\textnormal{div}\left(\abs{\nabla u}^{p-2}\nabla u\right)-\nabla\abs{\nabla u}^{p-2}\cdot\nabla u+\left(\abs{\nabla u}^{p-2}u_{ij}\right)_ix_j.
		\end{align*}
		For the second term in the left-hand side of \eqref{claim2}, we first compute
		\[
		f'(u)u_j=-\textnormal{div}\left(\abs{\nabla u}^{p-2}\nabla u_j\right)-(p-2)\,\textnormal{div}\left(\abs{\nabla u}^{p-4}u_{kj}u_k\nabla u\right),
		\]
		and then
		\begin{align*}
			f'(u)u_jx_j&=-\left(\abs{\nabla u}^{p-2}u_{ij}\right)_ix_j-(p-2)\left(\abs{\nabla u}^{p-4}u_{kj}u_ku_i\right)_ix_j.
		\end{align*}
		Hence, we have the following identity for the left-hand side in \eqref{claim2}:
		\begin{align*}
		&\hspace{-0.2cm}\textnormal{div}\left(\abs{\nabla u}^{p-2}\nabla \boldsymbol{c}\right)+f'(u)\boldsymbol{c}
		\\
		&\hspace{0.6cm}=2\,\textnormal{div}\left(\abs{\nabla u}^{p-2}\nabla u\right)
		-\nabla\abs{\nabla u}^{p-2}\cdot\nabla u-(p-2)\left(\abs{\nabla u}^{p-4}u_{kj}u_ku_i\right)_ix_j.
		\end{align*}
		Now, writting the last term in \eqref{claim2} as
		\begin{equation*}
		\begin{split}
		-\textnormal{div}\big((x\cdot\nabla\abs{\nabla u}^{p-2})\nabla u\big)
		&=-\nabla\abs{\nabla u}^{p-2}\cdot\nabla u
		-\left(\left(\abs{\nabla u}^{p-2}\right)_ju_i\right)_ix_j,
		\\
		&=-\nabla\abs{\nabla u}^{p-2}\cdot\nabla u
		-(p-2)\left(\abs{\nabla u}^{p-4} u_{kj} u_k u_i\right)_ix_j,
		\end{split}
		\end{equation*}
		we conclude claim \eqref{claim2}.
		
		At this point, using that $\nabla \boldsymbol{c}=\nabla u+D^2u\cdot x$, we write the third term in \eqref{stab2} as
		\begin{equation}\label{123}
		\begin{split}
		&(p-2)\,\textnormal{div}\left(\abs{\nabla u}^{p-4}(\nabla \boldsymbol{c}\cdot \nabla u)\nabla u\right)
		\\
		&\hspace{0.6cm}=(p-2)\,\textnormal{div}\left(\abs{\nabla u}^{p-2}\nabla u\right)
		+(p-2)\,\textnormal{div}\left(\abs{\nabla u}^{p-4} D^2u[x,\nabla u]\nabla u\right).
		\end{split}
		\end{equation}
		After adding up the two identities \eqref{claim2} and \eqref{123}, we exploit the cancellation of the last term in \eqref{claim2} with the last one in \eqref{123}, finishing the proof of \eqref{stab2}.
	\end{proof}
\end{lemma}

Combining the previous two lemmas, we obtain a new inequality for stable regular solutions. 

\begin{lemma}\label{lemma_stab3}
	Let $u$ be a stable regular solution of $-\plaplacian u =f(u)$ in $B_1\subset\R^n$, with $f\in C^1(\R)$, and $\eta$ a Lipschitz function with compact support in $B_1$. Then,
	\begin{equation}\label{78}
	\begin{split}
	&\hspace{-0.33cm}\int_{B_1}2\abs{\nabla u}^pr\,\eta_{r}\,\eta\,dx-p\int_{B_1}ru_r\abs{\nabla u}^{p-2}\nabla u\cdot\nabla(\eta^2) \,dx+(n-p)\int_{B_1}\abs{\nabla u}^p\eta^2\,dx
	\\
	&\hspace{3cm}\leq\int_{B_1}r^2u_r^2\Big\{\abs{\nabla u}^{p-2}\abs{\nabla\eta}^2+(p-2)\abs{\nabla u}^{p-4}\left(\nabla\eta\cdot\nabla u\right)^2
	\Big\} \,dx.
	\end{split}
	\end{equation}	
	
\end{lemma}
\begin{proof}
	We plug identity \eqref{stab2} into inequality \eqref{stab1}. Observe that \eqref{stab2} holds in the weak sense in $B_1\cap\{\abs{\nabla u}>\e\}$ and that all integrands in \eqref{stab1} are supported in $B_1\cap\{\abs{\nabla u}>\e\}$ thanks to the presence of the cut-off function $\phi_\e$. In this way, we obtain
	\begin{equation}\label{79}
	\begin{split}
	&\int_{B_1} p \,\diver\left(\abs{\nabla u}^{p-2}\nabla u\right)\boldsymbol{c}\eta^2\phi^2_\e\,dx
	\\
	&\hspace{1cm}\leq\int_{B_1} \boldsymbol{c}^2\Big\{\abs{\nabla u}^{p-2}\abs{\nabla\eta}^2+(p-2)\abs{\nabla u}^{p-4}\left(\nabla\eta\cdot\nabla u\right)^2	\Big\}\phi^2_\e\,dx+o(1),
	\end{split}
	\end{equation}
	where $o(1)$ goes to zero as $\e\downarrow0$.
	
	Now, using that $\nabla\boldsymbol{c}=\nabla u+D^2u\cdot x$, we use the following well-known Pohozaev-type identity in the weak sense in $\{\abs{\nabla u}>0\}$:
	\begin{equation*}
	\begin{split}
	&\hspace{-0.2cm}-\diver\big(\abs{\nabla u}^px-p\boldsymbol{c}\abs{\nabla u}^{p-2}\nabla u\big)\eta^2+(n-p)\abs{\nabla u}^p\eta^2
	\\
	&\hspace{2cm}=-\left(x\cdot\nabla\abs{\nabla u}^p\right)\eta^2-n\abs{\nabla u}^p\eta^2+p\abs{\nabla u}^{p-2}\nabla u\cdot\left\{\nabla u+D^2u\cdot x\right\}\eta^2
	\\
	&\hspace{3cm}+p\,\diver\big(\abs{\nabla u}^{p-2}\nabla u\big)\boldsymbol{c}\eta^2+(n-p)\abs{\nabla u}^p\eta^2
	\\
	&\hspace{2cm}=p \,\diver\left(\abs{\nabla u}^{p-2}\nabla u\right)\boldsymbol{c}\eta^2.
	\end{split}
	\end{equation*}
	Plugging this identity in \eqref{79}, we obtain
	\begin{equation*}
	\begin{split}
	&\hspace{-0.5cm}-\int_{B_1}\diver\big(\abs{\nabla u}^px-p\boldsymbol{c}\abs{\nabla u}^{p-2}\nabla u\big)\eta^2\phi^2_\e\,dx
	+(n-p)\int_{B_1}\abs{\nabla u}^p\eta^2\phi^2_\e\,dx
	\\
	&\hspace{1.2cm}\leq\int_{B_1}\boldsymbol{c}^2\Big\{\abs{\nabla u}^{p-2}\abs{\nabla \eta}^2+(p-2)\abs{\nabla u}^{p-4}\left(\nabla\eta\cdot\nabla u\right)^2\Big\}\phi^2_\e\,dx+o(1).
	\end{split}
	\end{equation*}
	We now integrate by parts the first term, controlling all the terms containing $\nabla\phi_\e$ with $o(1)$ by using \eqref{cutoff_2}, \eqref{funct_space2}, and dominated convergence. Recalling also that $\boldsymbol{c}=ru_r$, we get
	\begin{equation*}
	\begin{split}
	&\hspace{-0.315cm}\int_{B_1}2\abs{\nabla u}^pr\eta_{r}\eta\phi^2_\e\,dx-p\int_{B_1}ru_r\abs{\nabla u}^{p-2}\nabla u\cdot\nabla(\eta^2)\phi^2_\e\,dx 
	+(n-p)\int_{B_1}\abs{\nabla u}^p\eta^2\phi^2_\e\,dx
	\\
	&\hspace{1.2cm}\leq\int_{B_1}r^2u_r^2\Big\{\abs{\nabla u}^{p-2}\abs{\nabla\eta}^2+(p-2)\abs{\nabla u}^{p-4}\left(\nabla\eta\cdot\nabla u\right)^2
	\Big\}\phi^2_\e\,dx +o(1).
	\end{split}
	\end{equation*}
	
	Now, sending $ \e\downarrow0 $ and using dominated convergence (recall that $\nabla u\in L^p(B_1)$), all the four integrals above tend to the four integrals in \eqref{78}, but integrated on $B_1\cap\{\abs{\nabla u}>0\}$. Finally, notice that, since all integrands vanish in $\{\abs{\nabla u}=0\}$, it is equivalent to integrate over $B_1\cap\{\abs{\nabla u}>0\}$ or over $B_1$.
\end{proof}

We now take $\eta=\varphi\zeta$ in the previous lemma, with $\zeta$ a cut-off function in $B_{3\rho/2}$. In this way, $\varphi$ is not required to have compact support.

\begin{lemma}
	Let $u$ be a stable regular solution of $-\plaplacian u =f(u)$ in $B_1\subset\R^n$, with $f\in C^1(\R)$, and $\varphi$ a Lipschitz function in $B_1$. Then, for every $\rho\in(0,2/3) $ we have
	\begin{equation}\label{stab4}
	\begin{split}
	&\int_{B_\rho}2\abs{\nabla u}^pr\varphi_{r}\varphi\,dx-p\int_{B_\rho}ru_{r}\abs{\nabla u}^{p-2}\nabla u\cdot\nabla(\varphi^2)\,dx+(n-p)\int_{B_\rho}\abs{\nabla u}^p\varphi^2 \,dx
	\\
	&\hspace{1.7cm}-\int_{B_\rho}r^2u_r^2\abs{\nabla u}^{p-2}\abs{\nabla\varphi}^2\,dx-(p-2)\int_{B_\rho}r^2u_r^2\abs{\nabla u}^{p-4}\left(\nabla u\cdot\nabla\varphi\right)^2\,dx
	\\
	&\hspace{1cm}\leq C\int_{B_{3\rho/2}\setminus B_\rho}\abs{\nabla u}^p\Big\{\varphi^2+r\abs{\varphi\nabla\varphi}+r^2\abs{\nabla\varphi}^2\Big\}\,dx,
	\end{split}
	\end{equation}
	where $C$ is a constant depending only on $n$ and $p$.
\end{lemma}
\begin{proof}
	We choose $\eta=\varphi\zeta$ in Lemma \ref{lemma_stab3}, where $\zeta\in C^{\infty}_c(B_1)$ is a cut-off function such that $ \zeta\equiv1 $ in $ B_\rho $ and $ \zeta\equiv0 $ outside $B_{3\rho/2}$. In this way, $ \eta_{|\partial B_1}=0 $ and no assumption is required on $ \varphi $, except from Lipschitz regularity. We compute
\begin{equation*}
\begin{aligned}
	\eta_{r}\eta&=\varphi_{r}\varphi\zeta^2+\varphi^2\zeta_{r}\zeta;
	\\
	\nabla(\eta^2)&=\zeta^2\nabla(\varphi^2)+\varphi^2\nabla(\zeta^2);
	\\
	\abs{\nabla\eta}^2&=\zeta^2\abs{\nabla\varphi}^2+\varphi^2\abs{\nabla\zeta}^2+2\varphi\zeta\nabla\varphi\cdot\nabla\zeta;
	\\
	\nabla\eta\cdot\nabla u&=\zeta\nabla\varphi\cdot\nabla u+\varphi\nabla\zeta\cdot\nabla u.
\end{aligned}
\end{equation*}
Hence, inequality \eqref{78} becomes
\begin{equation*}
	\begin{split}
	&\hspace{-0.1cm}\int_{B_1}2\abs{\nabla u}^pr\left(\varphi_{r}\varphi\zeta^2+\varphi^2\zeta_{r}\zeta\right)\,dx
	-p\int_{B_1}r\,u_{r}\abs{\nabla u}^{p-2}\nabla u\cdot\left(\zeta^2\nabla(\varphi^2)+\varphi^2\nabla(\zeta^2)\right)\,dx
	\\
	&\hspace{6cm}+(n-p)\int_{B_1}\abs{\nabla u}^p\varphi^2\zeta^2\,dx
	\\
	&\hspace{2cm}\leq\int_{B_1}r^2u_r^2\Big\{\abs{\nabla u}^{p-2}\big(\zeta^2\abs{\nabla\varphi}^2+\varphi^2\abs{\nabla\zeta}^2+2\,\varphi\,\zeta\nabla\varphi\cdot\nabla\zeta\big)
	\\
	&\hspace{5.5cm}+(p-2)\abs{\nabla u}^{p-4}\left(\zeta\nabla\varphi\cdot\nabla u+\varphi\nabla\zeta\cdot\nabla u\right)^2\Big\}\,dx.
	\end{split}
\end{equation*}

Finally, we collect all the integrals containing $\nabla\zeta$ and we put them on the right-hand side, noticing that these are integrals over $B_{3\rho/2}\setminus B_\rho$ and that $\abs{\nabla\zeta}\leq \frac{C}{\rho}\leq\frac{3C}{2r}$ in this set. Each of the remaining integrals is written as an integral over $B_\rho$ (where $\zeta=1$) plus an integral over $B_{3\rho/2}\setminus B_\rho$, which is placed on the right-hand side. In this way, we conclude \eqref{stab4}.
\end{proof}

We finally make the choice $\varphi(x)=r^{-a/2}$ after regularizing this power function to make it Lipschitz. Note that the integrals in the following statement are finite since $a<n$ and $\abs{\nabla u}$ is locally bounded (recall that $u\in C^1$ in $B_1$).

\begin{lemma}\label{lemma_stab5}
	 Let $u$ be a stable regular solution of $-\plaplacian u =f(u)$ in $B_1\subset\R^n$, with $f\in C^1(\R)$, and $a<n$. Then, for every $\rho\in(0,2/3) $ we have
	 \begin{equation}\label{stab5}
	 	\begin{split}
	 	&\hspace{-1cm}(n-p-a)\int_{B_\rho}r^{-a}\abs{\nabla u}^p\,dx+\left(pa-\frac{a^2}{4}\right)\int_{B_\rho}r^{-a}\abs{\nabla u}^{p-2}u_{r}^2\,dx
	 	\\
	 	&\hspace{3cm}-\frac{a^2}{4}\left(p-2\right)\int_{B_\rho}r^{-a}\abs{\nabla u}^{p-4}u_{r}^4\,dx
	 	\\
	 	&\hspace{0.5cm}\leq C\int_{B_{3\rho/2}\setminus B_\rho} r^{-a}\abs{\nabla u}^p\,dx,
	 	\end{split}
	 \end{equation}
	 for some positive constant $C$ depending only on $n$, $p$, and $a$.
\end{lemma}
\begin{proof}
	In \eqref{stab4}, we choose $ \varphi(x)=r^{-a/2} $, with $a<n$. Since $\varphi$ may not be Lipschitz, the computations must be done with a regularization of $\varphi$ in a small neighborhood of the origin. For instance, for $a>0$ (that is how we will apply later the lemma), we take $\varphi_\e(x):=\min\{r^{-a/2},\e^{-a/2} \}$. Since all terms in the proof are given by integrable functions (recall that $u\in C^1(\overline{B}_\rho)$), by dominated convergence we can let $\e\downarrow0$ in all the integrals. For this reason, we directly write the computations with $\varphi$ instead of $\varphi_\e$.
	
	Now, \eqref{stab5} follows directly from \eqref{stab4} using that
	\begin{align*}
		\nabla\varphi&=-\frac{a}{2} r^{-\frac{a}{2}-2}x, &\qquad \varphi_{r}&=-\frac{a}{2} r^{-\frac{a}{2}-1}, 
		\\
		\nabla(\varphi^2)&=-a r^{-a-2}x, &\qquad \abs{\nabla\varphi}^2&=\frac{a^2}{4}r^{-a-2},
	\end{align*}
	and observing that $ r=\abs{x}$ is bounded from above and from below in $\in B_{3\rho/2}\setminus B_\rho $.
\end{proof}

We are now ready to give the

\begin{proof}[Proof of Lemma \ref{lemma_stability}]
	We assume $n>p$, as well as that $n$ and $p$ satisfy~\eqref{dimension}.	
	Choosing $a=n-p<n$ in \eqref{stab5}, we get
	\begin{equation*}
	\begin{split}
	&\hspace{-0.8cm}\left(p(n-p)-\left(\frac{n-p}{2}\right)^2\right)\int_{B_\rho}r^{p-n}\abs{\nabla u}^{p-2}u_{r}^2\,dx
	\\
	&\hspace{0.7cm}-(p-2)\left(\frac{n-p}{2}\right)^2\int_{B_\rho}r^{p-n}\abs{\nabla u}^{p-4}u_{r}^4\,dx \leq C\int_{B_{3\rho/2}\setminus B_\rho}r^{p-n}\abs{\nabla u}^p\,dx.
	\end{split}
	\end{equation*}
	
Now, if $ p\geq2 $, we use $ u_{r}^2\leq\abs{\nabla u}^2 $ to obtain
\begin{equation*}
\begin{split}
&\hspace{-0.5cm}\left(p(n-p)-\left(\frac{n-p}{2}\right)^2-(p-2)\left(\frac{n-p}{2}\right)^2\right)\int_{B_\rho}r^{p-n}\abs{\nabla u}^{p-2}u_{r}^2\,dx 
\\ 
&\hspace{5cm}\leq C\rho^{p-n}\int_{B_{3\rho/2}\setminus B_\rho}\abs{\nabla u}^p\,dx.
\end{split}
\end{equation*}
Instead, if $ p\in(1,2) $, we have
\begin{equation*}
\left(p(n-p)-\left(\frac{n-p}{2}\right)^2\right)\int_{B_\rho}r^{p-n}\abs{\nabla u}^{p-2}u_{r}^2\,dx \leq C\rho^{p-n}\int_{B_{3\rho/2}\setminus B_\rho}\abs{\nabla u}^p\,dx.
\end{equation*}

	Summarizing, if $n$ and $p$ satisfy \eqref{dimension}, then
	\[
	\int_{B_\rho}r^{p-n}\abs{\nabla u}^{p-2}u_{r}^2\,dx
	\leq C\rho^{p-n}\int_{B_{3\rho/2}\setminus B_\rho}\abs{\nabla u}^p\,dx,
	\]
	where $C$ depends only on $n$ and $p$.
\end{proof}

\section{Higher integrability}\label{sec_higher}
In this section we prove the key steps towards Theorem \ref{thm_higher_int}. We divide the section into two separate lemmata.
The first one establishes bounds on the weighted integrals of the second derivatives of a stable solution in terms of the $L^p$-norm of its gradient. This lemma, as well as the second one on higher integrability for the gradient, will be improved later in Section~\ref{sec_interior} by establishing the better control with the $L^1$-norm of $u$ in the right-hand side stated in~\eqref{intro_IIder_p>2} and~\eqref{intro_IIder_p<2}.

\begin{lemma}\label{lemma_II_der}
	Let $u$ be a stable regular solution of $ -\Delta_p u=f(u) $ in $ B_1\subset\R^n $, with $f\in C^1(\R)$ nonnegative. 
	
	Then,
	\begin{equation}
	\label{IIder_p>2}
	\int_{B_{1/2}}\abs{\nabla u}^{p-2}\lvert D^2u\rvert\,dx\leq C\norm{\nabla u}_{L^p(B_1)}^{p-1} \qquad \text{if}\,\,p\geq2 
	\end{equation}
	and
	\begin{equation}
	\label{IIder_p<2}
	\int_{B_{1/2}}\lvert D^2u\rvert\,dx\leq C\norm{\nabla u}_{L^p(B_1)} \qquad \text{if}\,\,p\in(1,2),
	\end{equation}
	where $C$ depends only on $n$ and $p$.
\end{lemma}
\begin{proof} First, by the regularity results commented in Remark \ref{rmk_integrability} and Stampacchia's theorem --- see \cite[Theorem 6.19]{LL} --- it is enough to show \eqref{IIder_p>2}-\eqref{IIder_p<2} with the integrals in their left-hand sides computed over $B_{1/2}\cap\{\abs{\nabla u}>0\}$ instead of $B_{1/2}$. In fact, as we saw in the remark, the value of the integrals does not change after changing the integration set in this way.
	
	We choose a nonnegative function $ \eta\in C^\infty_c(B_1)$ such that $ \eta\equiv1 $ in $ B_{1/2} $. Theorem \ref{thm_SZ_stability} gives
	\begin{equation}\label{789}
	\int_{B_1\cap\{|\nabla u|>0\}}\abs{\nabla u}^{p-2}\left\{\abs{\nabla_T\abs{\nabla u}}^2+\abs{A}^2\abs{\nabla u}^2 \right\}\eta^2\,dx
	\leq C\int_{B_1}\abs{\nabla u}^p\,dx,	
	\end{equation}
	where $C$ depends only on $ n $ and $ p $. We now use the well-known identity 
	\begin{equation*}
	\sum_{i,j}u_{ij}^2-\sum_i\left(\sum_ju_{ij}\frac{u_j}{\abs{\nabla u}}\right)^2=\abs{\nabla_T\abs{\nabla u}}^2+\abs{A}^2\abs{\nabla u}^2 ,
	\end{equation*}
	which is proved in \cite[Lemma 2.1]{SZ} --- see also \cite[Proposition 2.2]{C}. Plugging it in~\eqref{789}, we control\footnote{For this, express the symmetric matrix $D^2u$ in an orthonormal basis having $\nabla u/\abs{\nabla u}$ as last vector and take into account that the last column will be equal to the last row.}  the weighted integral of all the second derivatives of $u$, except from $D^2u[\nu,\nu]$, with a different constant $C$ in the right-hand side of \eqref{789}. Here and throughout the proof, $\nu$ denotes the normal vector to the level sets of $u$, i.e., $\nu=\nabla u/\abs{\nabla u}$, which is well-defined where $\nabla u\neq0$. As a consequence, we get
	\begin{equation}\label{84}
	\int_{B_{1}\cap\{\abs{\nabla u}>0 \}}\abs{\nabla u}^{p-2}\lvert D^2u-D^2u[\nu,\nu]\nu\otimes\nu\rvert^2\eta^2\,dx\leq C\int_{B_1}\abs{\nabla u}^p\,dx.
	\end{equation}
	
	\medskip
	
	\textit{Case 1.}
	Let us first assume $p\geq2$. We need to show that
	\begin{equation}\label{new_II_p>2}
	\int_{B_{1/2}\cap\{\abs{\nabla u}>0 \}}\abs{\nabla u}^{p-2}\lvert D^2u\rvert\,dx\leq C\norm{\nabla u}_{L^p(B_1)}^{p-1}.
	\end{equation}
	From \eqref{84}, using Hölder's inequality (since $p\geq2$), we deduce
	\begin{equation}\label{00}
	\int_{B_{1}\cap\{\abs{\nabla u}>0 \}}\abs{\nabla u}^{p-2}\lvert D^2u-D^2u[\nu,\nu]\nu\otimes\nu\rvert\eta^2\,dx\leq C\left(\int_{B_1}\abs{\nabla u}^p\,dx\right)^\frac{p-1}{p}.
	\end{equation}
		
	Now, we use that $ -\Delta_p u\geq0$ in order to bound the weighted integral of~$\lvert D^2u[\nu,\nu]\rvert$. 
	We claim that
	\begin{equation}\label{claim3}
	\int_{B_{1}\cap\{\abs{\nabla u}>0 \}}\abs{\nabla u}^{p-2}\lvert D^2u[\nu,\nu]\rvert\eta^2\,dx
	\leq C \norm{\nabla u}_{L^p(B_1)}^{p-1},
	\end{equation}
	where $C$ depends only on $n$ and $p$. To establish this, we start from the identity
	\begin{equation}\label{99}
	\begin{aligned}
	\plaplacian u &=\diver\left(\abs{\nabla u}^{p-2}\nabla u\right)=\abs{\nabla u}^{p-2}\Delta u+(p-2)\abs{\nabla u}^{p-2}\sum_{i,j} \frac{u_{ij}u_iu_j}{\abs{\nabla u}^2}
	\\
	&= (p-1)\abs{\nabla u}^{p-2}D^2u[\nu,\nu]+\abs{\nabla u}^{p-2}\textnormal{tr}\left(D^2u_{|\nu^\bot\otimes\nu^\bot}\right).
	\end{aligned}
	\end{equation}
	From this, we express $(p-1)\abs{\nabla u}^{p-2}D^2u[\nu,\nu]$ as the difference of two terms, we then take absolute values, and we integrate over $B_{1/2}\cap\{\abs{\nabla u}>0\}$. Using that
	\begin{equation}\label{trace_ineq}
	\lvert\textnormal{tr}\left(D^2u_{|\nu^\bot\otimes\nu^\bot}\right)\rvert\leq
	C\lvert D^2u-D^2u[\nu,\nu]\nu\otimes\nu\rvert
	\end{equation}
	we conclude
	\begin{equation}\label{97}
	\begin{split}
	&(p-1)\int_{B_{1/2}\cap\{\abs{\nabla u}>0 \}}\abs{\nabla u}^{p-2}\lvert D^2u[\nu,\nu]\rvert\,dx
	\\
	&\hspace{1cm}\leq
	\int_{B_{1/2}\cap\{\abs{\nabla u}>0 \}}\abs{\plaplacian u}\,dx
	+
	C\int_{B_{1/2}\cap\{\abs{\nabla u}>0 \}}\abs{\nabla u}^{p-2}\lvert D^2u-D^2u[\nu,\nu]\nu\otimes\nu\rvert\,dx.
	\end{split}
	\end{equation}
	We now use Lemma \ref{positive_plap} (which is based on the fact that $-\plaplacian u\geq0$) and Hölder's inequality to get
	\begin{equation}\label{01}
	\int_{B_{1/2}}\abs{\Delta_pu}\,dx\leq C \norm{\nabla u}_{L^p(B_1)}^{p-1}.
	\end{equation}
	At this point, using \eqref{97}, \eqref{00}, and \eqref{01}, we obtain \eqref{claim3}.
	Finally, combining~\eqref{00} with~\eqref{claim3}, the proof of \eqref{new_II_p>2} is finished. 
	
	\medskip
	
	\textit{Case 2.}
	Now we assume $ p\in(1,2) $. We need to prove that
	\begin{equation}\label{new_II_p<2}
	\int_{B_{1/2}\cap\{\abs{\nabla u}>0 \}}\lvert D^2u\rvert\,dx\leq C\norm{\nabla u}_{L^p(B_1)}.
	\end{equation} 
	First, choosing again a function $\eta\in C^\infty_c(B_1)$ with $\eta\equiv1$ in $ B_{1/2} $, we observe that
	\begin{equation*}
	\begin{split}
	&\int_{B_{1}\cap\{\abs{\nabla u}>0 \}}\lvert D^2u-D^2u[\nu,\nu]\nu\otimes\nu\rvert\eta^2\,dx
	\\
	&\hspace{0.5cm}\leq
	\left(\int_{B_{1}\cap\{\abs{\nabla u}>0 \}}\abs{\nabla u}^{p-2}\lvert D^2u-D^2u[\nu,\nu]\nu\otimes\nu\rvert^2\eta^2\,dx\right)^\frac12\left(\int_{B_1}\abs{\nabla u}^{2-p}\eta^2\,dx\right)^\frac12.
	\end{split}
	\end{equation*}
	Now, using \eqref{84} and Hölder's inequality on the last integral (note that $p/(2-p)>1$), we obtain
	\begin{equation}\label{85}
	\int_{B_{1}\cap\{\abs{\nabla u}>0 \}}\lvert D^2u-D^2u[\nu,\nu]\nu\otimes\nu\rvert\eta^2\,dx
	\leq
	C\left(\int_{B_1}\abs{\nabla u}^p\,dx\right)^\frac1p.
	\end{equation}
	
	At this point, in order to prove \eqref{new_II_p<2}, it only remains to control the $L^1 $-norm of $D^2u[\nu,\nu]$. Since we are working in the set of regular points $ \{x\in B_1: \abs{\nabla u}>0\} $, we can write identity~\eqref{99} as
	\begin{equation}\label{100}
	\abs{\nabla u}^{2-p}\plaplacian u=(p-1)D^2u[\nu,\nu]+\textnormal{tr}\left(D^2u_{|\nu^\bot\otimes\nu^\bot}\right).
	\end{equation}
	Having $\plaplacian u\leq0 $ and $ \lvert\textnormal{tr}\left(D^2u_{|\nu^\bot\otimes\nu^\bot}\right)\rvert\leq C\lvert D^2u-D^2u[\nu,\nu]\nu\otimes\nu\rvert$ as in Case 1, from \eqref{85} and \eqref{100} we deduce that
	\begin{equation}\label{up}
	\int_{B_{1}\cap\{\abs{\nabla u}>0 \}}D^2u[\nu,\nu]\eta^2\,dx\leq C\norm{\nabla u}_{L^p(B_1)}.
	\end{equation}
	Next, since $\Delta u=\textnormal{tr}\left(D^2u_{|\nu^\bot\otimes\nu^\bot}\right)+D^2u[\nu,\nu]$, we can reformulate \eqref{100} as
	\[
	-\abs{\nabla u}^{2-p}\plaplacian u = (2-p) D^2u[\nu,\nu] - \Delta u.
	\]
	After multiplying this identity by $\eta^2$ and integrating it over $B_1\cap\{\abs{\nabla u}>0\}$, we can use~\eqref{up}, an integration by parts, and Hölder's inequality, to get
	\begin{equation}\label{96}
	\begin{split}
	\int_{B_1\cap\{\abs{\nabla u}>0\}}\abs{\nabla u}^{2-p}\abs{\plaplacian u}\,\eta^2\,dx 
	&=-\int_{B_1\cap\{\abs{\nabla u}>0\}}\abs{\nabla u}^{2-p}\plaplacian u\,\eta^2\,dx 
	\\
	&= (2-p)\int_{B_1\cap\{\abs{\nabla u}>0\}}D^2u[\nu,\nu]\eta^2\,dx-\int_{B_1}\Delta u\,\eta^2\,dx
	\\
	&\leq
	C\norm{\nabla u}_{L^p(B_1)}+\int_{B_1}\nabla u\cdot\nabla\eta^2\,dx\leq C\norm{\nabla u}_{L^p(B_1)}.
	\end{split}
	\end{equation}
	
	Finally, isolating $D^2u[\nu,\nu]$ in identity \eqref{100}, taking absolute values, and using the bounds \eqref{trace_ineq}, \eqref{85}, and~\eqref{96}, we conclude
	\begin{equation*}
	\begin{split}
	&\int_{B_{1}\cap\{\abs{\nabla u}>0 \}}\lvert D^2u[\nu,\nu]\rvert\eta^2\,dx
	\\
	&\hspace{1.2cm}\leq C\int_{B_{1}\cap\{\abs{\nabla u}>0 \}}\lvert D^2u-D^2u[\nu,\nu]\nu\otimes\nu\rvert\eta^2\,dx
	+\int_{B_1\cap\{\abs{\nabla u}>0 \}}\abs{\nabla u}^{2-p}\abs{\plaplacian u}\,\eta^2\,dx
	\\
	&\hspace{1.2cm}\leq C\norm{\nabla u}_{L^p(B_1)}.
	\end{split}
	\end{equation*}
	This and \eqref{85} finish the proof of \eqref{new_II_p<2}.
\end{proof}

The following lemma establishes a control over the $L^{p+\gamma}$-norm of the gradient of $u$ in terms of its $L^p$-norm in a larger ball, for some $\gamma>0$, proving inequality~\eqref{intro_Lp_controls_p+gamma}.
An alternative proof of this result, relying on the Michael-Simon and Allard inequality, is presented in Appendix~\ref{appendix_alt_proof}. The ranges of values of $\gamma$ obtained in both proofs will be given later in Remark~\ref{rmk_gamma}.

\begin{lemma}\label{lemma_Lpgamma}
	Let $u$ be a stable regular solution of $ -\Delta_pu=f(u) $ in $ B_1\subset\R^n $, with $ f\in C^1(\R) $ nonnegative. 
	Then,
	\begin{equation}\label{19}
	\norm{\nabla u}_{L^{p+\gamma}(B_{1/2})}\leq C\norm{\nabla u}_{L^p(B_1)},
	\end{equation}
	for some positive constants $\gamma$ and $C$ depending only on $n$ and $p$.
\end{lemma}

\begin{proof}
	Without loss of generality, we can assume that 
	\[\norm{\nabla u}_{L^p(B_1)}=1,\] 
	which will be useful later, in Step 3 of the proof. Let $ \eta\in C^\infty_c(B_1)$ be a nonnegative function such that $ \eta\equiv1 $ in $ B_{1/2} $. We divide the proof of~\eqref{19} in three steps. 
	
	\medskip
	
	\textbf{Step 1.} We claim that we can control the $L^1$-norm of the $(p+1)$-Laplacian of $u$ in terms of a constant depending only on $n$ and $p$ (after multiplying by the cut-off function $\eta^2$), i.e.,
	\begin{equation}\label{div_bound}
	\int_{B_{1}}\abs{\textnormal{div}\left(\abs{\nabla u}^{p-1}\nabla u\right)}\eta^2\,dx\leq C.
	\end{equation}
	
	For this, we first note that $\lvert\textnormal{div}\left(\abs{\nabla u}^{p-1}\nabla u\right)\rvert\eta^2$ is integrable in $B_1$. 
	This follows from
	\begin{equation*}
	\begin{split}
	\int_{B_1}\abs{\nabla u}^{p-1}\lvert D^2u\rvert\eta^2\,dx&=\int_{B_1}\abs{\nabla u}^{\frac{p-2}{2}}\lvert D^2u\rvert\abs{\nabla u}^{\frac{p}{2}}\eta^2\,dx
	\\
	&\leq C\left(\int_{B_1}\abs{\nabla u}^{p-2}\lvert D^2u\rvert^2\eta^2\,dx\right)^{\frac12}\left(\int_{B_1}\abs{\nabla u}^p\,dx\right)^\frac12,
	\end{split}
	\end{equation*}
	and the fact that $ \nabla u\in W^{1,2}_{\sigma,\text{loc}}(B_1) $, where the $W^{1,2}_{\sigma}$-norm is defined in \eqref{weighted_norm} --- see Remark~\ref{rmk_integrability}.
	In this way, the left-hand side of \eqref{div_bound} is well-defined and we will be able to integrate by parts later in \eqref{101}.
	
	Next, we use the identities
	\begin{equation}\label{identity_div}
	\textnormal{div}\left(\abs{\nabla u}^{p-1}\nabla u\right)=\abs{\nabla u}\plaplacian u+\abs{\nabla u}^{p-1}D^2u[\nu,\nu]
	\end{equation}
	and, from \eqref{99}, 
	\begin{equation}\label{90}
	(p-1)\abs{\nabla u}^{p-1}D^2u[\nu,\nu]=\abs{\nabla u}\plaplacian u-\abs{\nabla u}^{p-1}\textnormal{tr}\left(D^2u_{|\nu^\bot\otimes\nu^\bot}\right).
	\end{equation}
	Observe that we can combine them to obtain
	\begin{equation}\label{identity_div_new}
	(p-1)\textnormal{div}\left(\abs{\nabla u}^{p-1}\nabla u\right)=p\abs{\nabla u}\plaplacian u-\abs{\nabla u}^{p-1}\textnormal{tr}\left(D^2u_{|\nu^\bot\otimes\nu^\bot}\right).
	\end{equation}	
	After multiplying identity \eqref{90} by $\eta^2$ and integrating it over $B_1$, we use that $\plaplacian u\leq0 $ and
	$\lvert\textnormal{tr}\left(D^2u_{|\nu^\bot\otimes\nu^\bot}\right)\rvert\leq C\lvert D^2u-D^2u[\nu,\nu]\nu\otimes\nu\rvert$ to deduce
	\begin{equation*}
	(p-1)\int_{B_{1}}\abs{\nabla u}^{p-1}D^2u[\nu,\nu]\eta^2\,dx
	\leq C\int_{B_{1}}\abs{\nabla u}^{p-1}\lvert D^2u-D^2u[\nu,\nu]\nu\otimes\nu\rvert\eta^2\,dx.
	\end{equation*}
	This, together with \eqref{84} and Hölder's inequality, gives
	\begin{equation}\label{up_p+gamma}
	\int_{B_{1}}\abs{\nabla u}^{p-1}D^2u[\nu,\nu]\eta^2\,dx
	\leq 
	C\int_{B_{1}}\abs{\nabla u}^{p-1}\lvert D^2u-D^2u[\nu,\nu]\nu\otimes\nu\rvert\eta^2\,dx
	\leq C.
	\end{equation}
	Now, multiplying identity \eqref{identity_div} by $ \eta^2 $ and integrating it over $ B_1 $ we get
	\begin{equation}\label{101}
	\begin{split}
	-\int_{B_1}\abs{\nabla u}\plaplacian u\,\eta^2\,dx&=\int_{B_1}\abs{\nabla u}^{p-1}D^2u[\nu,\nu]\eta^2\,dx-\int_{B_{1}}\,\textnormal{div}\left(\abs{\nabla u}^{p-1}\nabla u\right)\eta^2\,dx
	\\
	&\leq
	C+\int_{B_{1}}\abs{\nabla u}^{p-1}\nabla u\cdot\nabla \eta^2\,dx\leq C,
	\end{split}
	\end{equation}
	where we used \eqref{up_p+gamma} to control the first term in the right-hand side of the identity and we integrated by parts the second one.
	
	Finally, taking absolute values in \eqref{identity_div_new}, using $\lvert\textnormal{tr}\left(D^2u_{|\nu^\bot\otimes\nu^\bot}\right)\rvert\leq C\lvert D^2u-D^2u[\nu,\nu]\nu\otimes\nu\rvert$ and $-\plaplacian u\geq0 $, we obtain 
	\begin{equation*}
	\begin{split}
	&(p-1)\int_{B_{1}}\abs{\textnormal{div}\left(\abs{\nabla u}^{p-1}\nabla u\right)}\eta^2\,dx
	\\
	&\hspace{2cm}\leq
	-p\int_{B_1}\abs{\nabla u}\plaplacian u\,\eta^2\,dx + C\int_{B_{1}}\abs{\nabla u}^{p-1}\lvert D^2u-D^2u[\nu,\nu]\nu\otimes\nu\rvert\eta^2\,dx.
	\end{split}
	\end{equation*}
	This, together with \eqref{up_p+gamma} and \eqref{101}, proves \eqref{div_bound}.
	
	\medskip
	
	\textbf{Step 2.} We show that
	\begin{equation}\label{89}
	\int_{\{u=t\}\cap B_{1/2}}\abs{\nabla u}^{p}\,d\mathcal{H}^{n-1}\leq C \qquad \text{for a.e.}\,\,t\in\R,
	\end{equation}
	where $C$ is a constant depending only on $n$ and $p$. For this purpose, we claim that
	\begin{equation}\label{4}
	\begin{split}
	\int_{\{u=t\}\cap B_{1/2}}\abs{\nabla u}^{p}\,d\mathcal{H}^{n-1}&\leq
	\int_{\{u=t\}\cap B_{1}}\abs{\nabla u}^{p}\eta^2\,d\mathcal{H}^{n-1}
	\\
	&=
	-\int_{\{u>t\}\cap B_{1}}\textnormal{div}\left(\abs{\nabla u}^{p-1}\nabla u\,\eta^2\right)\,d\mathcal{H}^{n-1}.
	\end{split}
	\end{equation}
	This, together with \eqref{div_bound}, will give \eqref{89}. Hence, we are left with proving the equality in~\eqref{4}.
	
	To show \eqref{4}, note that we cannot use Sard's theorem (since the function $u$ is only $C^{1,\vartheta}$) to ensure the regularity of the level sets in~\eqref{4}. For this reason, we must be careful in integrating by parts in the equality of \eqref{4}.
	To do it properly, we take a smooth approximation $s\longmapsto K_\varepsilon(s)$ of the characteristic function of~$\R_+$ in such a way that $ K'_\varepsilon(s) \rightharpoonup^* \delta_0 $ as $ \varepsilon\to0 $. Recall that in Step 1 we have proved that $\lvert\textnormal{div}\left(\abs{\nabla u}^{p-1}\nabla u\right)\rvert\eta^2\in L^1(B_1)$.
	Hence, for every $ t\in\R $,
	\begin{equation*}
	\begin{split}
	-\int_{B_{1}}K_\varepsilon(u-t)\,\textnormal{div}\left(\abs{\nabla u}^{p-1}\nabla u\,\eta^2\right)\,dx&=
	\int_{B_{1}}K'_\varepsilon(u-t)\abs{\nabla u}^{p+1}\eta^2\,dx
	\\
	&\hspace{-2cm}=\int_\R K'_\varepsilon(\tau-t)\left(\int_{\{u=\tau\}\cap B_{1}}\abs{\nabla u}^{p}\eta^2\,d\mathcal{\mathcal{H}}^{n-1}\right)\,d\tau,
	\end{split}
	\end{equation*}
	where, in the last step, we applied the coarea formula --- see \cite[Lemma A.2]{CFRS} for a precise statement.
	Now, since $u\in C^{1,\vartheta}(B_1)$ we have that $\int_{B_{1}}\abs{\nabla u}^{p+1}\eta^2\,dx<+\infty$ and hence the function $\tau\mapsto\int_{\{u=t \}\cap B_{1}}\abs{\nabla u}^p\eta^2\,d\mathcal{H}^{n-1}$ belongs to $L^1(\R)$. Thus, since a.e. $t\in\R$ is a Lebesgue point for this $L^1(\R)$ function, for such values of $t$ we let $\e\downarrow0$ and conclude the equality in \eqref{4}.
	
	\medskip
	
	\textbf{Step 3.}
	Here we conclude the proof of the lemma. By the Sobolev-Poincaré inequality, there exists an exponent $ m>p $, depending only on~$n$ and~$p$, such that
	\begin{equation}\label{m_ineq}
	\left(\int_{B_1}\abs{u-\overline{u}}^m\, dx\right)^\frac1m
	\leq C\left(\int_{B_{1}}\abs{\nabla u}^p\, dx\right)^\frac1p = C,
	\end{equation}
	where $\overline{u}=\ave_{B_1}u\,dx.$
	Thus, using the coarea formula, we obtain
	\begin{equation}\label{55}
	\int_\R dt\int_{\{u=t\}\cap B_1\cap\{\abs{\nabla u}>0\}}\abs{t-\overline{u}}^m\abs{\nabla u}^{-1}dx
	=\int_{B_1\cap\{\abs{\nabla u}>0\}}\abs{u-\overline{u}}^m\,dx
	\leq C.
	\end{equation}
	
	Since $ m>p $, we can choose $ q>1 $ and $\theta\in\big(0,\frac{1}{p+1}\big)$ such that
	\[
	\frac{m}{q}=\frac{1-\theta}{\theta},
	\]
	and thus $m\theta-q(1-\theta)=0$. Defining $ h(t):=\max\{1,\abs{t-\overline{u}}\}$, and using the coarea formula and Hölder's inequality, we get
	\begin{equation*}
	\begin{split}
	\int_{B_{1/2}}\abs{\nabla u}^{(p+1)(1-\theta)}\,dx&=\int_\R\, dt\int_{\{u=t\}\cap B_{1/2}\cap\{\abs{\nabla u}>0\}} h(t)^{m\theta-q(1-\theta)}\abs{\nabla u}^{p(1-\theta)-\theta}\,d\mathcal{\mathcal{H}}^{n-1}
	\\
	&\leq
	\left(\int_\R dt\int_{\{u=t\}\cap B_{1}\cap\{\abs{\nabla u}>0\}} h(t)^{m}\abs{\nabla u}^{-1}\,d\mathcal{\mathcal{H}}^{n-1}\right)^\theta
	\\
	&\hspace{1.5cm}\cdot
	\left(\int_\R dt\int_{\{u=t\}\cap B_{1/2}\cap\{\abs{\nabla u}>0\}} h(t)^{-q}\abs{\nabla u}^{p}\,d\mathcal{\mathcal{H}}^{n-1}\right)^{1-\theta}.
	\end{split}
	\end{equation*}
	We observe that $\int_\R h(t)^{-q}dt<+\infty$ since $q>1$, and hence by \eqref{89} we have that
	\[
	\int_\R dt\,h(t)^{-q}\int_{\{u=t\}\cap B_{1/2}}\abs{\nabla u}^{p}\,d\mathcal{\mathcal{H}}^{n-1}\leq C \int_\R h(t)^{-q}\,dt \leq C.
	\]
	On the other hand, using the definition of $h(t)$ and \eqref{55}, we have
	\begin{equation*}
	\begin{split}
	&\int_\R dt\int_{\{u=t\}\cap B_{1}\cap\{\abs{\nabla u}>0\}} h(t)^{m}\abs{\nabla u}^{-1}\,d\mathcal{\mathcal{H}}^{n-1}
	\\
	&\hspace{2cm}\leq
	\int_{\overline{u}-1}^{\overline{u}+1} dt\int_{\{u=t\}\cap B_{1}\cap\{\abs{\nabla u}>0\}}\abs{\nabla u}^{-1}\,d\mathcal{\mathcal{H}}^{n-1}+C
	\leq
	\abs{B_1}+C.
	\end{split}
	\end{equation*}	
	Therefore, we have established that
	\[
	\int_{B_{1/2}}\abs{\nabla u}^{(p+1)(1-\theta)}\,dx\leq C.
	\]
	Since $ \theta\in\big(0,\frac{1}{p+1}\big) $, we have that $(p+1)(1-\theta)>p$. This proves \eqref{19} for some $\gamma>0$ depending only on $n$ and $p$. 
\end{proof}

\begin{remark}\label{rmk_gamma}
	We report here the explicit values of $\gamma$ allowed in Lemma \ref{lemma_Lpgamma}, as given by the previous proof and by the alternative proof in Appendix \ref{appendix_alt_proof}. In the previous proof of Lemma~\ref{lemma_Lpgamma}, for $n>p$ we have $m=\frac{np}{n-p}$ in \eqref{m_ineq}, while any finite $m$ is allowed when $n\leq p$. Letting $q\downarrow1$ we deduce that every $\gamma$ satisfying
	\begin{equation*}
	\begin{cases}
	\gamma<\frac{p^2}{np+n-p} \qquad &\text{when}\,\,n>p,
	\\
	\gamma<1 \qquad &\text{when}\,\,n\leq p,
	\end{cases}
	\end{equation*}
	is allowed.
	
	Instead, the alternative proof of Lemma \ref{lemma_Lpgamma} presented in Appendix \ref{appendix_alt_proof}, which is based on the Michael-Simon and Allard inequality, gives 
	\begin{equation*}
	\begin{cases}
	\gamma=\frac{2(p-1)}{n-1}\qquad&\text{when}\,\,n\geq4,
	\vspace{2mm} \\ 
	\gamma=\frac{2(p-1)}{3}\qquad&\text{when}\,\,n\leq3.
	\end{cases}
	\end{equation*}	
	
	Note that, depending on the values of $n$ and $p$, the highest value of $\gamma$ can be given either by the first proof or by the second one\footnote{Take for instance $n=4$ and $p=2$, and then $n=3$ and $p=3/2$.}.
	However, we recall that, for the scope of the present paper, we only need to know that inequality \eqref{19} --- and hence later also~\eqref{L1controls_p} --- hold for some $\gamma>0$ depending only on $n$ and $p$.
\end{remark}

\section{The weighted radial derivative controls the full gradient}\label{sec_doubling}
In this section we prove a lemma that gives, under a doubling assumption on $\abs{\nabla u}^pdx$, a control of the $ L^p $-norm of~$ \nabla u $ by the weighted $L^2$-norm of the radial derivative of $u$. It is stated next. We will use it as a key tool in the proof of Theorem \ref{thm_interior}, combining it with the key estimate \eqref{stability_estimate} for stable solutions that we proved in Section \ref{sec_stability}.
\begin{lemma}\label{lemma_doubling}
	Let $u$ be a stable regular solution of $-\Delta_p u=f(u)$ in $ B_2\subset\R^n $, with $f\in C^1(\R)$ nonnegative. Assume that
	\[
	\int_{B_{1}}\abs{\nabla u}^p\,dx\geq\delta\int_{B_2}\abs{\nabla u}^p\,dx,
	\]
	for some $ \delta>0 $. Then,
	\[
	\int_{B_{3/2}}\abs{\nabla u}^p\,dx\leq C_\delta\int_{B_{3/2}\setminus B_{1}}\abs{\nabla u}^{p-2}u_r^2\,dx,
	\]
	where $C_\delta$ depends only on $n$, $p$, and $\delta$.
\end{lemma}
\begin{proof}
	Assume the result to be false. Then, there exist $\delta>0$ and a sequence of stable regular solutions $ (u_k)_k $ to $ -\Delta_p u_k=f_k(u_k) $ in $B_2$, for a sequence of nonnegative $C^1$ nonlinearities $f_k$, such that
	\begin{equation}\label{doubling_assumption}
	\int_{B_1}\abs{\nabla u_k}^p\,dx\geq\delta\int_{B_2}\abs{\nabla u_k}^p\,dx,
	\end{equation}
\begin{equation}\label{57}
	\int_{B_{3/2}}\abs{\nabla u_k}^p\,dx=1,
\end{equation}
and
	\begin{equation}\label{radial_der_to0}
	\lim_{k\to\infty}\int_{B_{3/2}\setminus B_1}\abs{\nabla u_k}^{p-2}(\partial_r u_k)^2\,dx=0.
	\end{equation}	
	As a consequence of \eqref{doubling_assumption} and \eqref{57}, we have
	\begin{equation}\label{lp}
	\int_{B_2}\abs{\nabla u_k}^p\,dx
	\leq\frac1\delta\int_{B_1}\abs{\nabla u_k}^p\,dx
	\leq \frac1\delta\int_{B_{3/2}}\abs{\nabla u_k}^p\,dx=\frac1\delta\leq C,
	\end{equation}
	and, using Lemma \ref{lemma_Lpgamma} (properly rescaled) together with a covering argument, we deduce
	\begin{equation}\label{uk_p+gamma}
	\int_{B_{3/2}}\abs{\nabla u_k}^{p+\gamma}\,dx\leq C
	\end{equation}
	for some $\gamma>0$.
	
	Hence, the sequence of $p$-superharmonic functions
	\[
	v_k:= u_k - \ave_{B_2}u_k\,dx
	\]
	satisfies, thanks to H\"older's and Poincaré inequalities combined with \eqref{lp},
	\begin{equation}\label{vk_bounds1}
	\norm{v_k}_{L^1(B_2)}\leq C\norm{v_k}_{L^p(B_2)}\leq C.
	\end{equation}
	In addition, thanks to \eqref{57}, \eqref{lp}, \eqref{uk_p+gamma}, and \eqref{radial_der_to0}, the sequence $(v_k)_k$ satisfies
	\begin{equation}\label{vk_bounds2}
	\norm{\nabla v_k}_{L^p(B_{3/2})}= 1, \quad
	\norm{\nabla v_k}_{L^p(B_2)}\leq C, \quad
	\norm{\nabla v_k}_{L^{p+\gamma}(B_{3/2})}\leq C,
	\end{equation}
	and
	\begin{equation}\label{vk_bounds3}
	\lim_{k\to\infty}\int_{B_{3/2}\setminus B_1}\abs{\nabla v_k}^{p-2}(\partial_r v_k)^2\,dx=0.
	\end{equation}
	By \eqref{vk_bounds1} and \eqref{vk_bounds2}, the sequence $(v_k)_k$ is bounded in $W^{1,p}(B_{3/2})$, thus it weakly converges in $W^{1,p}(B_{3/2})$ to a limit $v_\infty\in W^{1,p}(B_{3/2})$, up to a subsequence.
	Next, we claim that, up to a subsequence, we have
	\begin{equation}\label{68}
	v_k\longrightarrow v_\infty \qquad \text{strongly in}\,\,W^{1,p}(B_{3/2}).
	\end{equation}
	
	To prove this we have to distinguish whether $p\geq2$ or $p\in(1,2)$.
	
	\medskip
	
	\textit{Case 1.} We start by assuming $p\geq2$.		
	Besides $\norm{\nabla v_k}_{L^p(B_{3/2})}=1$,
	we also have, using Lemma \ref{lemma_II_der} (properly rescaled and combined with a covering argument) and \eqref{vk_bounds2}, that
	\[
	\norm{\abs{\nabla v_k}^{p-2}D^2v_k}_{L^1(B_{3/2})}\leq C\norm{\nabla v_k}^{p-1}_{L^p(B_2)}\leq C,
	\]
	where $C$ does not depend on $k$. Therefore, we can apply Lemma \ref{lemma_convergence} (rescaled from $B_1$ to~$B_{3/2}$) to every component of $\nabla v_k$ and deduce that, up to subsequence, $\nabla v_k$ converges strongly in $L^{p-1}(B_{3/2})$ to $\nabla v_\infty$, where $v_\infty$ was the $W^{1,p}(B_{3/2})$ weak limit of $(v_k)_k$. Combining this with the Poincaré inequality, we deduce that, up to a subsequence,
	\begin{equation}\label{23}
	v_k\longrightarrow v_\infty \qquad \text{strongly in}\,\,W^{1,p-1}(B_{3/2}).
	\end{equation}
	In addition, by \eqref{vk_bounds2} we know that
	\begin{equation*}
	\norm{\nabla v_k}_{L^{p+\gamma}(B_{3/2})}\leq C,
	\end{equation*}
	for some $ \gamma>0 $. Since $p-1<p<p+\gamma$, combining this bound with \eqref{23}, claim \eqref{68} follows for $p\geq2$.
	
	\medskip
	
	\textit{Case 2.}
	When $p\in(1,2)$, from Lemma \ref{lemma_II_der} (properly rescaled and combined with a covering argument) and \eqref{vk_bounds2} it follows that 
	\[
	\int_{B_{3/2}}\lvert D^2v_k\rvert\,dx\leq C\norm{\nabla v_k}_{L^p(B_2)}\leq C.
	\]
	Thus, the sequence $(\nabla v_k)_k$ is uniformly bounded in $ W^{1,1}(B_{3/2})$. From the compactness of the immersion $ W^{1,1}(B_{3/2})\hookrightarrow L^1(B_{3/2}) $, up to a subsequence we have that
	\begin{equation*}
	v_k\longrightarrow v_\infty \qquad \text{strongly in}\,\,W^{1,1}(B_{3/2}).
	\end{equation*}
	We finally combine this convergence with the last bound in \eqref{vk_bounds2} to deduce \eqref{68} for $p\in(1,2)$.
	
	\medskip

	In the rest of the proof we do not need to distinguish whether $p$ is above or below~$2$.
	As a consequence of \eqref{vk_bounds2}, \eqref{vk_bounds3}, and \eqref{68}, we claim that $ v_\infty\in W^{1,p}(B_{3/2}) $, $ v_\infty $ is weakly $ p $-superharmonic,  
	$$
	\norm{\nabla v_\infty}_{L^p(B_{3/2})}=1, \qquad\text{and}\qquad \partial_{r} v_\infty=0 \quad\text{ a.e. in } B_{3/2}\setminus \overline{B}_1.
	$$
	To prove the last statement, we use the inequality 
	$$
	\left| |a|^{p-2}(a\cdot \sigma)^{2}-|b|^{p-2}(b\cdot \sigma)^{2} \right| \leq C(|a|+|b|)^{p-1}|a-b| \qquad\text{for all }a \text{ and }b
	 \text{ in }\R^{n}
	$$
	and all $\sigma\in\R^{n}$ satisfying $|\sigma|=1$, where $C$ is a constant depending only on $n$ and $p$. The inequality is easily proved controling the derivative of the function $c\in\R^{n} \mapsto |c|^{p-2}(c\cdot \sigma)^{2}$ along the segment joining $a$ and $b$ by $C(|a|+|b|)^{p-1}$. From it, we see that
	\begin{equation*}
\begin{split}
	&\int_{B_{3/2}\setminus B_1}\left| \abs{\nabla v_k}^{p-2}(\partial_r v_k)^2 - \abs{\nabla v_\infty}^{p-2}(\partial_r v_\infty)^2 \right| \,dx\\
	& \hspace{1cm} \leq C \int_{B_{3/2}\setminus B_1}\big( \abs{\nabla v_k}+\abs{\nabla v_\infty}\big)^{p-1}  \abs{\nabla v_k -\nabla v_\infty} \,dx,
\end{split}
\end{equation*}
and using H\"older's inequality and that $\nabla v_{k}\to \nabla v_{\infty}$ in $L^p(B_{3/2})$, we get that the above integrals converge to $0$. As a consequence, by \eqref{vk_bounds3}, we conclude that $\partial_{r }v_\infty=0 $ a.e. in $B_{3/2}\setminus \overline{B}_1$.

	Let now $v$ be the $p$-harmonic function in $B_{3/2}\setminus \overline{B}_1$ such that $v-v_\infty\in W^{1,p}_0(B_{3/2}\setminus \overline{B}_1)$. Note that $v$ exists and it is unique --- see, e.g., \cite[Theorem 2.16]{Lind}. 
	By the weak comparison principle for the $p$-Laplacian, we have that
	\[
		v_\infty\geq v \qquad \text{in}\,\, B_{3/2}\setminus \overline{B}_1.
		\]
	Now, since $v\in C^0(\overline{B}_{11/8}\setminus B_{9/8})$, we have that
	\[
	c_0:=\inf_{B_{11/8}\setminus B_{9/8}}v_\infty \geq \min_{\overline{B}_{11/8}\setminus B_{9/8}}v>-\infty.
	\]
	In addition, by the $0$-homogeneity of $v_\infty$ in $B_{3/2}\setminus\overline{B}_1$, we will have
	\[
	c_0 = \inf_{B_{11/8}\setminus B_{9/8}}v_\infty = \inf_{B_{1/8}(x_0)}v_\infty = \inf_{B_{1/16}(x_0)}v_\infty
	\]
	for some point $x_0\in\partial B_{10/8}$. Now, the weak Harnack inequality applied to the $p$-superharmonic function $v_\infty-c_0$, which is nonnegative in $B_{1/8}(x_0)$ --- see, e.g., \cite[Corollary 3.18]{Lind} --- gives that $v_\infty-c_0\equiv0$ in $B_{1/16}(x_0)$, since its infimum in the ball of half radius $B_{1/16}(x_0)$ is zero. 
	As a consequence, we can now repeat the same argument with $x_0$ replaced by any other point on $B_{1/16}(x_0)\cap\partial B_{10/8}$. This, together with a covering argument, leads to $v_\infty\equiv c_0$ in $B_{\frac{10}{8}+\frac{1}{16}}\setminus\overline{B}_{\frac{10}{8}-\frac{1}{16}}$.
	In addition, since $\partial_{r} v_\infty\equiv0$ in $B_{3/2}\setminus\overline{B}_1$, we deduce that
		\begin{equation}\label{claim4}
		v_\infty\equiv c_0\,\,\,\text{a.e. in}\,\,B_{3/2}\setminus \overline{B}_1. 
		\end{equation}
	
	In particular, we have proved that ${v_\infty}_{|\partial B_1}=c_0$. Thus, by the weak comparison principle for $p$-superharmonic functions we get $v_\infty\geq c_0$ in $B_{1}$. Combining this with \eqref{claim4}, and again by the weak Harnack inequality for $p$-superharmonic functions, we obtain that
	\[
	v_\infty\equiv c_0 \qquad \text{a.e. in}\,\,B_{3/2}.
	\]
	This is a contradiction with $\norm{\nabla v_\infty}_{L^p(B_{3/2})}=1$ and finishes the proof of the lemma.
\end{proof}

\section{Interior estimates. Global estimates in convex domains}\label{sec_interior}
We report here the following general lemma from \cite[Lemma 3.2]{CFRS}. We refer to the cited paper for its proof.

\begin{lemma}[Cabré, Figalli, Ros-Oton, Serra \cite{CFRS}]\label{lemma_seq}
	Let $ \{a_j\}_{j\geq0} $ and $ \{b_j\}_{j\geq0} $ be two sequences of nonnegative numbers satisfying $ a_0\leq M $, $ b_0\leq M $,
	\begin{equation*}
	b_j\leq b_{j-1}\qquad \text{and} \qquad a_j+b_j\leq La_{j-1} \qquad \text{for all}\,\,\, j\geq1,
	\end{equation*}
	and
	\begin{equation*}
	\text{if}\quad\, a_j\geq\frac12a_{j-1} \quad\,\text{then}\quad\, b_j\leq L\left(b_{j-1}-b_j\right) \qquad \text{for all}\,\,\, j\geq1,
	\end{equation*}
	for some positive constants $M$ and $L$. Then there exist $ \theta\in(0,1) $ and $ C>0 $, depending only on $ L $, such that
	\[
	b_j\leq CM\theta^j \qquad \text{for all}\,\,\, j\geq0.
	\]
\end{lemma}

With this lemma and a new interpolation inequality adapted to the $p$-Laplacian --- that we establish in Appendix \ref{appendix_lemmata}; see Proposition \ref{prop5.2} --- we are able to prove our main result.

\begin{proof}[Proof of Theorem \ref{thm_interior}]	
	We split the proof into three steps. 
		
\medskip

\textbf{Step 1.} With the aim of proving the higher integrability estimate \eqref{L1controls_p}, we first claim that in every dimension $n$
	\begin{equation}\label{claim1}
	\norm{\nabla u}_{L^p(B_{1/2})}\leq C \norm{u}_{L^1(B_1)},
	\end{equation}	
	where $C$ depends only on $n$ and $p$. In order to prove this inequality, we have to distinguish whether $p$ is above or below $2$. 
	
	\medskip
	
	\textit{Case 1.} We start by assuming $p\geq2$ and we pick $\gamma>0$ given by Lemma \ref{lemma_Lpgamma}. By interpolation and that lemma, we obtain
\begin{equation*}
\begin{split}
\norm{\nabla u}_{L^p(B_{1/2})}
&\leq \norm{\nabla u}_{L^{p+\gamma}(B_{1/2})}^\theta
\norm{\nabla u}_{L^{p-1}(B_{1/2})}^{1-\theta}
\\
&\leq C\norm{\nabla u}_{L^{p}(B_1)}^\theta
\norm{\nabla u}_{L^{p-1}(B_{1/2})}^{1-\theta}
\end{split}
\end{equation*}
for some $\theta\in(0,1)$ depending only on $n$ and $p$.
From this estimate, we use Young's inequality and apply the new interpolation inequality of Proposition \ref{prop5.2} in each cube (of universal size) in a family of cubes covering $B_{1/2}$ and contained in $B_{3/4}$) --- note that the proposition can be applied by the regularity shown in Remark \ref{rmk_integrability}. Adding the inequalities obtained in each cube, we deduce
\begin{equation*}
\begin{split}
\norm{\nabla u}_{L^p(B_{1/2})}^{p-1} &\leq \widetilde{\e}\norm{\nabla u}_{L^{p}(B_1)}^{p-1} +
\frac{C}{\widetilde{\e}^a}\norm{\nabla u}_{L^{p-1}(B_{1/2})}^{p-1}
\\
&\leq\widetilde{\e}\norm{\nabla u}_{L^{p}(B_1)}^{p-1} + \frac{C}{\widetilde{\e}^a}\left(\e \int_{B_{3/4}}\abs{\nabla u}^{p-2}\lvert D^2u\rvert\,dx+ C\e^{1-p}\norm{u}^{p-1}_{L^{p-1}(B_{3/4})}\right),
\end{split}
\end{equation*}
for every $\e>0$ and $\widetilde{\e}>0$, where $a>0$ and $C$ depend only on~$n$ and~$p$.
Now, we  apply \eqref{IIder_p>2} (after a covering and rescaling argument, to replace $B_{1/2}$ by $B_{3/4}$ on its left-hand side) and choose $\e$ such that $C\e/\widetilde{\e}^a=\widetilde{\e}$, to get
\begin{equation*}
\begin{split}
\norm{\nabla u}_{L^p(B_{1/2})}^{p-1} 
&\leq\widetilde{\e}\norm{\nabla u}_{L^{p}(B_1)}^{p-1} + \frac{C}{\widetilde{\e}^a}\left(\e \norm{\nabla u}^{p-1}_{L^p(B_1)}+ C\e^{1-p}\norm{u}^{p-1}_{L^{p-1}(B_{1})}\right)
\\
&\leq 2\widetilde{\e}\norm{\nabla u}^{p-1}_{L^p(B_1)} + C\widetilde{\e}^{-b}\norm{u}^{p-1}_{L^{p-1}(B_{1})},
\end{split}
\end{equation*}
for every $\widetilde{\e}>0$, where $b>0$ depends only on $n$ and $p$. 

We now apply this estimate to the functions $u_{\rho,y}(x):=u(y+\rho x)$ for $B_\rho(y)\subset B_1$; note that all these functions are stable regular solutions to a $p$-Laplace equation with nonnegative nonlinearity.  We obtain that
\begin{align*}
\rho^{\frac{n}{p-1}+p} & \int_{B_{\rho/2}(y)}|\nabla u|^p\,dx
\\
\leq  &
C \tilde\e^{\frac{p}{p-1}} \rho^{\frac{n}{p-1}+p}\int_{B_{\rho}(y)}|\nabla u|^p\,dx
+ \frac{C}{\tilde\e^{\frac{bp}{p-1}}}\left(\int_{B_{\rho}(y)}|u|^{{p-1}}\,dx\right)^{\frac{p}{p-1}}
\\
\leq &
C \tilde\e^{\frac{p}{p-1}} \rho^{\frac{n}{p-1}+p}\int_{B_{\rho}(y)}|\nabla u|^p\,dx
+ \frac{C}{\tilde\e^{\frac{bp}{p-1}}}\left(\int_{B_{1}}|u|^{{p-1}}\,dx\right)^{\frac{p}{p-1}}
\end{align*}
for every $\tilde\e>0$, where the constant $C$ depends only on $n$ and $p$.
By Lemma \ref{lemma_simon} applied with
$\mathcal S (B):=\|\nabla u\|_{L^p(B)}^p$ and $\beta=\frac{n}{p-1}+p$, we deduce that for every $p\geq2$ and dimension $n$ we have
\begin{equation}\label{44}
\norm{\nabla u}_{L^p(B_{1/2})}\leq C \norm{u}_{L^{p-1}(B_{1})}.
\end{equation}

If $p=2$ estimate \eqref{44} coincides with our claim \eqref{claim1}. Therefore, we may assume $p>2$ and apply \eqref{44} to $u-\overline{u}$, where $\overline{u}:=\ave_{B_{1}}u\,dx$. By interpolation, we get
\begin{equation*}
\begin{split}
\norm{\nabla u}_{L^p(B_{1/2})}
&\leq C \norm{u-\overline{u}}_{L^{p-1}(B_{1})}
\\
&\leq C\norm{u-\overline{u}}_{L^p(B_{1})}^\nu \,
\norm{u-\overline{u}}_{L^1(B_{1})}^{1-\nu},
\end{split}
\end{equation*}
for some $\nu\in(0,1)$ depending only on $p$.
Using Young's inequality and the fact that $\norm{u-\overline{u}}_{L^1(B_{1})}\leq2\norm{u}_{L^1(B_{1})}$, we get
\[
\norm{\nabla u}_{L^p(B_{1/2})}
\leq
\e\norm{u-\overline{u}}_{L^p(B_{1})} +
\frac{C}{\e^{\widetilde{a}}}\norm{u}_{L^1(B_{1})}
\]
for some $\widetilde{a}>0$ depending only on $p$.
Finally, Poincaré inequality gives
\[
\norm{\nabla u}_{L^p(B_{1/2})}
\leq
C\e\norm{\nabla u}_{L^p(B_{1})} +
\frac{C}{\e^{\widetilde{a}}}\norm{u}_{L^1(B_{1})},
\]
for every $\e>0$, where $C$ depends only on $n$ and $p$. Applying this estimate to rescaled  and translated versions of $u$, as before, by Lemma \ref{lemma_simon} we conclude \eqref{claim1} for $p>2$. 

\medskip

\textit{Case 2.} Let us assume now $p\in(1,2)$. Similarly to the previous case, by interpolation and Lemma \ref{lemma_Lpgamma}, we obtain
\begin{equation*}
\begin{split}
\norm{\nabla u}_{L^p(B_{1/2})} 
&\leq \norm{\nabla u}_{L^{p+\gamma}(B_{1/2})}^{\tau}
\norm{\nabla u}_{L^{1}(B_{1/2})}^{1-\tau}
\\
&\leq C\norm{\nabla u}_{L^{p}(B_1)}^{\tau}
\norm{\nabla u}_{L^{1}(B_{1/2})}^{1-\tau}
\end{split}
\end{equation*}
for some $\tau\in(0,1)$ and $C$ depending only on $n$ and $p$. Now, by Young's inequality and the interpolation inequality \cite[Theorem 7.28]{GT} we deduce that
\begin{equation*}
\begin{split}
\norm{\nabla u}_{L^p(B_{1/2})} &\leq \widetilde{\e}\norm{\nabla u}_{L^{p}(B_1)} +
\frac{C}{\widetilde{\e}^c}\norm{\nabla u}_{L^{1}(B_{1/2})}
\\
&\leq\widetilde{\e}\norm{\nabla u}_{L^{p}(B_1)} + \frac{C}{\widetilde{\e}^c}\left(\e \int_{B_{1/2}}\lvert D^2u\rvert\,dx+ C\e^{-1}\norm{u}_{L^{1}(B_{1/2})}\right),
\end{split}
\end{equation*}
for every $\e>0$ and $\widetilde{\e}>0$, where $c>0$ and $C$ depend only on~$n$ and~$p$.
Applying~\eqref{IIder_p<2} and choosing $\e$ such that $C\e/\widetilde{\e}^c=\widetilde{\e}$, we get
\begin{equation*}
\begin{split}
\norm{\nabla u}_{L^p(B_{1/2})} 
&\leq\widetilde{\e}\norm{\nabla u}_{L^{p}(B_1)} + \frac{C}{\widetilde{\e}^c}\left(\e \norm{\nabla u}_{L^p(B_1)}+ C\e^{-1}\norm{u}_{L^{1}(B_{1})}\right)
\\
&\leq 2\widetilde{\e}\norm{\nabla u}_{L^p(B_1)} + C\widetilde{\e}^{-d}\norm{u}_{L^{1}(B_{1})},
\end{split}
\end{equation*}
where $d>0$ depends only on $n$ and $p$. From this inequality applied at all scales (as before), by Lemma \ref{lemma_simon} we deduce \eqref{claim1} for every $p\in(1,2)$ and dimension~$n$.

\medskip

At this point, since we proved \eqref{claim1} for every $p>1$, we can combine it with Lemma~\ref{lemma_Lpgamma} to get the bound
\[
\norm{\nabla u}_{L^{p+\gamma}(B_{1/4})}\leq C\norm{u}_{L^1(B_1)}.
\]
Finally, using a rescaling and covering argument we deduce the same estimate in $B_{1/2}$, i.e., \eqref{L1controls_p}.

In Step 1 of Proof of Lemma \ref{lemma_Lpgamma} we proved $\norm{\Delta_{p+1}u}_{L^1(B_{1/2})}\leq C \norm{\nabla u}_{L^p(B_1)}$. This estimate, rescaled and combined with \eqref{claim1}, leads to
\[
\norm{\Delta_{p+1}u}_{L^1(B_{1/4})}\leq C \norm{u}_{L^1(B_1)}.
\]
After a covering, we conclude \eqref{pplus1}.

\medskip
	
\textbf{Step 2.} The rest of the proof is devoted to show the interior $C^\alpha$ apriori estimate \eqref{C.alfa}. For this we may assume $n>p$ (by Remark \ref{artificial}), as well as that $n$ and $p$ satisfy \eqref{dimension}, in order to use the stability estimate~\eqref{stability_estimate} in Lemma~\ref{lemma_stability}. 

We first prove the bound
\begin{equation}\label{seminormCalpha}
\seminorm{u}_{C^\alpha(\overline{B}_{1/16})}\leq C\norm{\nabla u}_{L^p(B_{3/4})}
\end{equation}
for the $C^\alpha$ semi-norm.

For this, let $\rho\in(0,1)$ and define the quantities
\[
\mathcal{D}(\rho):=\rho^{p-n}\int_{B_\rho}\abs{\nabla u}^p\,dx \qquad \text{and} \qquad  \mathcal{R}(\rho):=\int_{B_\rho}r^{p-n}\abs{\nabla u}^{p-2}u_r^2\,dx.
\]
By the stability inequality \eqref{stability_estimate}, for every $ \rho\in(0,2/3) $ it holds that
\begin{equation}\label{s}
\mathcal{R}(\rho)\leq C\rho^{p-n}\int_{B_{3\rho/2}\setminus B_\rho}\abs{\nabla u}^p\,dx.
\end{equation}
On the other hand, if we assume that $\mathcal{D}(\rho)\geq\frac12\mathcal{D}(2\rho)$, a rescaling of Lemma \ref{lemma_doubling} applied to the function $ u(\rho \,\cdot) $ with $ \delta=\frac12 $, gives that
\begin{equation*}
\rho^{p-n}\int_{B_{3\rho/2}}\abs{\nabla u}^p\,dx
\leq C \int_{B_{3\rho/2}\setminus B_\rho}r^{p-n}\abs{\nabla u}^{p-2}u_r^2\,dx
= C\left\{\mathcal{R}(3\rho/2)-\mathcal{R}(\rho)\right\},
\end{equation*}
for some constant $ C $ depending only on $n$ and $p$. Combining this bound with \eqref{s} and using that $ \mathcal{R} $ is nondecreasing, we deduce that for every $ \rho<1/2 $ one has:
\begin{equation}\label{56}
\text{if}\quad\,\mathcal{D}(\rho)\geq\frac12\mathcal{D}(2\rho) \quad\,\text{then}\quad\, \mathcal{R}(\rho)\leq C\left\{\mathcal{R}(2\rho)-\mathcal{R}(\rho)\right\}.
\end{equation}

Hence, defining $ a_j:=\mathcal{D}(2^{-j-2}) $ and $b_j:=\mathcal{R}(2^{-j-2}) $, they satisfy the assumptions of Lemma~\ref{lemma_seq} by \eqref{s}, \eqref{56}, and the fact that $ \mathcal{R}$ is nondecreasing. From this, we deduce that
\[
\mathcal{R}(2^{-j-2})\leq CM\theta^j \qquad \text{for all}\,\,\,j\geq0
\]
for some $\theta\in(0,1)$, depending only on $n$ and $p$, and for $ M:=\abs{a_0}+\abs{b_0}\leq C\norm{\nabla u}_{L^p(B_{1/2})}^p $. Notice that we used again~\eqref{s} to bound $b_0$. Choosing $ \alpha>0 $ such that $ 2^{-p\alpha}=\theta $, this gives
\begin{equation}\label{s1}
\mathcal{R}(\rho)\leq C\rho^{p\alpha}\norm{\nabla u}_{L^p(B_{1/2})}^p \qquad \text{for every}\,\,\,\rho\leq1/4.
\end{equation}

We now apply \eqref{s1} to the function $ u_y(x):=u(x+y) $ for a given $ y\in B_{1/4} $. We use the following notation:
\[r_y(x)=\abs{x-y}\qquad\text{and}\qquad u_{r_y}(x)=\frac{x-y}{\abs{x-y}}\cdot\nabla u(x).\]
Observe that $ B_{1/2}(y)\subset B_{3/4} $. By \eqref{s1} (rescaled in order that $B_{1/2}(y)$ becomes $B_1(y)$) we obtain
\begin{equation}\label{567}
\rho^{p-n}\int_{B_\rho(y)}\abs{\nabla u}^{p-2}u_{r_y}^2\,dx\leq C\rho^{p\alpha}\int_{B_{3/4}}\abs{\nabla u}^p\,dx \qquad \text{for every}\,\,\,\rho\leq1/8.
\end{equation}
Now, given $ z\in B_{1/8} $, we average \eqref{567} with respect to $y\in B_{2\rho}(z)$, with $\rho\leq1/16$. More precisely, we first use the identity
\[
\abs{\nabla u}^2(x)=C_n\rho^{-n}\int_{B_\rho(x)\setminus B_{\rho/2}(x)}u_{r_y}^2(x)\,dy
\]
for $x\in B_\rho(z)\subset B_{1/4}$, where $C_n$ is a constant depending only on $n$. Since $B_\rho(x)\subset B_{1/4}$, this is easily checked as follows: 
\begin{equation*}
\begin{split}
\int_{B_\rho(x)\setminus B_{\rho/2}(x)}  u_{r_y}^2 (x) dy  
& = u_{x_i}(x)u_{x_j}(x) \int_{\rho/2}^{\rho} ds \int_{\partial B_s(x)}  s^{-2}(y_i-x_i)(y_j-x_j)\,d\sigma(y)  \\
& =  u_{x_i}^2(x) \int_{\rho/2}^{\rho} ds \int_{\partial B_s(x)} s^{-2}(y_i-x_i)^2\,d\sigma(y)\\
& = \sum_{i=1}^n  u_{x_i}^2(x) \int_{\rho/2}^{\rho}   \frac{|S^{n-1}|}{n} s^{n-1}\,ds= C_n^{-1} \rho^n |\nabla u(x)|^2.
\end{split}
\end{equation*}
From this identity and \eqref{567}, for every $z\in B_{1/8}$ and $\rho\leq1/16$ we have
\begin{equation*}
	\begin{split}
	\int_{B_{\rho}(z)}\abs{\nabla u}^p(x)\,dx&=C\rho^{-n}\int_{B_{\rho}(z)}dx\abs{\nabla u(x)}^{p-2}\int_{B_{\rho}(x)\setminus B_{\rho/2}(x)}dy\,u_{r_y}^2(x)
	\\
	&\hspace{-1.2cm}\leq C\rho^{-n}\int_{B_{2\rho}(z)}dy\int_{B_{\rho}(y)}dx\, \abs{\nabla u(x)}^{p-2}u_{r_y}^2(x)
	\leq
	C\rho^{n-p+p\alpha}\int_{B_{3/4}}\abs{\nabla u(x)}^p\,dx,
	\end{split}
\end{equation*}
where we have used that $y\in B_{2\rho}(z)\subset B_{1/4}$ as required to have \eqref{567}.
In this way, we obtain
\begin{equation*}
\rho^{p-n}\int_{B_{\rho}(z)}\abs{\nabla u}^p\,dx
\leq C\rho^{p\alpha}\int_{B_{3/4}}\abs{\nabla u}^p \,dx
\qquad \text{for every}\,\,\,\rho\leq1/16.
\end{equation*}
Since $ z\in B_{1/8} $ is arbitrary, by classical estimates in Morrey spaces --- see~\cite[Theorem~7.19]{GT} --- we deduce \eqref{seminormCalpha}.

\medskip

\textbf{Step 3.} Whenever $n$ and $p$ satisfy \eqref{dimension}, we can combine \eqref{seminormCalpha}, \eqref{claim1}, and a rescaling and covering argument, to obtain
\[
	\seminorm{u}_{C^\alpha(\overline{B}_{1/2})}\leq C\norm{u}_{L^1(B_1)}.
\]
This, together with the classical interpolation inequality
\[
\norm{u}_{L^\infty(B_{1/2})}\leq C\left(\seminorm{u}_{C^\alpha(\overline{B}_ {1/2})}+\norm{u}_{L^1(B_{1/2})}\right),
\]
finishes the proof of \eqref{C.alfa}.
\end{proof}

Combining Theorem \ref{thm_interior} with some known boundary estimates in strictly convex domains proved in \cite{CasSa}, we deduce Corollary \ref{cor_convex}.

\begin{proof}[Proof of Corollary \ref{cor_convex}]
	Since $\Omega$ is strictly convex, $f$ is positive, and $u$ is a regular solution of~\eqref{plap_Dir}, Proposition 3.1 in \cite{CasSa} ensures that\footnote{Even though \cite{CasSa} assumes $f$ to be smooth, their proof only uses that $f$ is locally Lipschitz in order to apply a comparison principle from \cite{DS}.
	} 
	\begin{equation}\label{boundary_est}
	\norm{u}_{L^\infty(\Omega\setminus K_{2\delta})}\leq C\norm{u}_{L^1(\Omega)},
	\end{equation}
	where $ K_{\delta}:=\left\{x\in\Omega: \textnormal{dist}(x,\partial\Omega)\geq\delta\right\} $, 
	and $C$ and $\delta$ are positive constants depending only on the domain $\Omega$.
	Using this bound, we can control $f(u)$ in $L^\infty$-norm in the set $\Omega\setminus K_{2\delta}$ by a constant depending only on $\Omega$, $p$, $f$, and $\norm{u}_{L^1(\Omega)}$. By interior and boundary regularity\footnote{See \cite[Theorem 1]{DiB} or \cite[Theorem 1]{T} for interior $ C^{1,\vartheta} $ regularity and \cite[Theorem 1]{L} for boundary regularity.}
	for problem \eqref{plap_Dir}, we deduce that $ \norm{\nabla u}_{L^p(\Omega\setminus K_{\delta})}\leq C$ (note that the closure of $\Omega\setminus K_\delta$ is contained in $\overline{\Omega}\setminus K_{2\delta}$), where $C$ depends only on $ \Omega $, $p$, $ f $, and $ \norm{u}_{L^1(\Omega)} $. On the other hand, by the interior estimate \eqref{L1controls_p} of Theorem \ref{thm_interior} we have that $\norm{\nabla u}_{L^p(K_\delta)}\leq C\norm{u}_{L^1(\Omega)}$ for some constant $C$ depending only on $\Omega$ and $p$ --- since $\delta$ was chosen depending only on $\Omega$. Thus, combining these two bounds we obtain~\eqref{Lp_convex_bound}.
	
	If $n$ and $p$ satisfy \eqref{dimension}, by the $C^\alpha$ interior estimate of Theorem \ref{thm_interior} we have that $\norm{u}_{L^\infty(K_\delta)}\leq C\norm{u}_{L^1(\Omega)}$, where $C$ depends again only on $\Omega$ and $p$. Combining this bound with \eqref{boundary_est}, we obtain~\eqref{convex_bound}.
\end{proof}

From Corollary \ref{cor_convex} and the $L^1$ uniform bounds for $u_\lambda$ proved by the third author in \cite{San,San1}, we deduce Theorem~\ref{thm_extremal_new} about extremal solutions.

\begin{proof}[Proof of Theorem \ref{thm_extremal_new}]
	Let $\{u_\lambda\}_\lambda$ be the family of smallest stable regular solutions to $(1.19)_{\lambda,p}$, as stated in Theorem \ref{thm_extremal_known}. The results in \cite{San,San1} ensure that we have a bound for the $L^1$-norms of $u_\lambda$ which is uniform in $\lambda$. Indeed, when $p\geq2$ this is proved in \cite[Proposition~3]{San1}, while for $p\in(1,2)$ the $L^1$ bound follows from\footnote{Since $p<2$ and $f$ satisfies \eqref{f_p_convexity}, we have that $f$ is convex and thus, in the notation of \cite{San}, it holds that $\tau_-\geq0>(p-2)(p-1)$. It then follows from \cite[Lemma 12]{San} that there exists some $\gamma\geq1/(p-1)$ for which inequality (21) in \cite[Proposition 10]{San} holds. One can therefore use this proposition, obtaining from it the uniform bound for the $L^1$-norms of $u_\lambda$ when $p\in(1,2)$. Note that \cite{San} assumes that $f\in C^2$ in order to define $\tau_-$ and to prove Lemma~12. However, since our nonlinearity $f\in C^1$ is assumed here to be convex, it is easy to see that the proof of \cite[Lemma 12]{San} still works even if $f$ is not $C^2$; indeed, one can even take $\e=0$ within the proof.} 
	\cite[Proposition 10]{San}.
	
	Given that~$\Omega$ is a strictly convex domain, from Corollary \ref{cor_convex} we deduce that
	\[
	\norm{\nabla u_\lambda}_{L^p(\Omega)}\leq C \qquad \text{for every}\,\,\,\lambda\in(0,\lambda^*),
	\]
	where $C$ depends only on $\Omega$, $p$, $f$, and $\norm{u_\lambda}_{L^1(\Omega)}$ (and in particular it is independent of $\lambda$).
	
	If we assume in addition \eqref{dimension}, Corollary \ref{cor_convex} also gives
	\[
	\norm{u_\lambda}_{L^\infty(\Omega)}\leq C\norm{u_\lambda}_{L^1(\Omega)} \qquad \text{for every}\,\,\,\lambda\in(0,\lambda^*),
	\]
	where $C$ depends only on $\Omega$ and $p$. 
	
	Now, thanks to the uniform $L^1$ bounds for $u_\lambda$ in \cite{San,San1} discussed above, we can pass to the limit as $\lambda\uparrow\lambda^*$ in both bounds, completing the proof of the energy estimates stated after Theorem~\ref{thm_extremal_new} and within Theorem~\ref{thm_extremal_new}, respectively.
\end{proof}

\section{Morrey and $L^q$ estimates}\label{sec_big_n}

In this section we show how our method gives also information about the integrability of stable solutions to \eqref{plap_eq} in higher dimensions. Here we prove the following stronger version of Theorem \ref{thm_big_n}.

\begin{theorem}\label{thm_big_n_morrey}
Let $n$ and $p$ satisfy~\eqref{big_dimension}, and $u$ be a stable regular solution to
$-\Delta_pu=f(u)$ in $B_1\subset\R^n$, with $ f\in C^1(\R)$. 

Then,
\begin{equation}\label{morrey_estimate_big_n}
\norm{u}_{M^{p+\frac{p^2}{\beta-p},\beta}(B_{1/4})}
+ \norm{\nabla u}_{M^{p,\beta}(B_{1/4})}
\leq C\norm{\nabla u}_{L^p(B_{1})} \qquad\text{for every}\,\,\,\beta\in(\overline{\beta},n),
\end{equation} 
where $C$ depends only on $n$, $p$, and $\beta$. The threshold exponent~$\overline{\beta}$ satisfies $p<\overline{\beta}<n$ and is defined by
\begin{equation}\label{beta_bar}
\overline{\beta}:=\begin{cases}
\begin{aligned}
&n-2-2\sqrt{\frac{n-1}{p-1}} \qquad &\text{if}\quad p\geq2, \\
&n-2(p-1)-2\sqrt{(p-1)^2+n-p} \qquad&\text{if}\quad p\in(1,2).
\end{aligned}
\end{cases}
\end{equation}

In addition, if $f$ is nonnegative, then $\norm{\nabla u}_{L^p(B_1)}$ can be replaced by $\norm{u}_{L^1(B_1)}$ in the right-hand side of \eqref{morrey_estimate_big_n}.
\end{theorem}
When $u$ is radially symmetric, we have:
\begin{equation}\label{radial_embedding}
\text{if}\,\,u\,\,\text{is radial and}\,\,\nabla u\in M^{p,\beta}(B_{1/4}),\,\,\text{then}\,\,u\in L^{s}(B_{1/8})\,\,\text{for all}\,\,s<\frac{np}{\beta-p}.
\end{equation}
This follows from \cite[Theorem 1.5]{CCh} after cutting-off $u$ outside $B_{1/8}$ to have compact support in $B_{1/4}$.
Using this embedding, we deduce from \eqref{morrey_estimate_big_n} that if $u$ is radial then $u\in L^s(B_{1/8})$ for every $s<\overline{s}$, where $\overline{s}$ is defined by
\begin{equation*}
\overline{s}:=
\begin{cases}
\begin{aligned}
&\frac{np}{n-2-2\sqrt{\frac{n-1}{p-1}}-p} \qquad &\text{if}\quad p\geq2, \\
&\frac{np}{n-2(p-1)-2\sqrt{(p-1)^2+n-p}-p} \qquad&\text{if}\quad p\in(1,2).
\end{aligned}
\end{cases}
\end{equation*}
For $p\geq2$ this integrability result is optimal under our standing hypothesis \eqref{big_dimension} on~$n$ and~$p$. Indeed, for every $p>1$, Capella, together with the first and third authors~\cite{CCaSa} proved that, if $n>p+4p/(p-1)$ and $u$ is radial, then
\[
\norm{u}_{L^s(B_{1})}\leq C\norm{\nabla u}_{W^{1,p}(B_1)}
\]
for every $s<np/(n-2-2\sqrt{\frac{n-1}{p-1}}-p)$, and that this limit exponent is sharp. Thus, in the radial case and for $p\geq2$, we recover from Theorem \ref{thm_big_n_morrey} and \eqref{radial_embedding} the sharp integrability estimates proved in~\cite{CCaSa}. For $p\in(1,2)$ instead, our method does not give the optimal results, not even in the radial case.

However, the embedding \eqref{radial_embedding} is false for non-radial functions, as shown recently by Charro and the first author in \cite{CCh}. 
For this reason, establishing whether $u\in L^s(B_{1/2})$ for every $s<np/(n-2-2\sqrt{\frac{n-1}{p-1}}-p)$ when $p\geq2$ remains an open question in the nonradial case.
\begin{proof}[Proof of Theorem \ref{thm_big_n_morrey}]
	We assume that $n$ and $p$ satisfy \eqref{big_dimension}. In Lemma \ref{lemma_stab5}, thanks to~\eqref{big_dimension} we can choose\footnote{One can easily check that, under assumption \eqref{big_dimension}, the intervals in \eqref{a_condition} are not empty and they have nonempty intersection with $(-\infty,n)$.
} an exponent $a<n$ such that
\begin{equation}\label{a_condition}
\begin{cases}
\begin{aligned}
& \frac{4p}{p-1}<a<2+2\sqrt{\frac{n-1}{p-1}} \qquad &\text{if}\quad p\geq2, \\
&  4p<a<2(p-1)+2\sqrt{(p-1)^2+n-p} \qquad&\text{if}\quad p\in(1,2),
\end{aligned}
\end{cases}
\end{equation}
and taking $\rho=1/4$ we obtain
\begin{equation}\label{new_stab5}
\begin{split}
&(n-p-a)\int_{B_{1/4}}r^{-a}\abs{\nabla u}^p\,dx+\left(pa-\frac{a^2}{4}\right)\int_{B_{1/4}}r^{-a}\abs{\nabla u}^{p-2}u_{r}^2\,dx
\\
&\hspace{2cm}
-\frac{a^2}{4}\left(p-2\right)\int_{B_{1/4}}r^{-a}\abs{\nabla u}^{p-4}u_{r}^4\,dx
\leq C\int_{B_{1/2}\setminus B_{1/4}} r^{-a}\abs{\nabla u}^p\,dx,
\end{split}
\end{equation}
where $C$ depends only on $n$, $p$, and $a$.

The third term in the left-hand side of \eqref{new_stab5} changes sign whether $p$ is above or below $2$. When it is positive --- i.e., when $p\in(1,2)$ --- we simply discard it. When it is nonpositive --- i.e., when $p\geq2$ --- we use $u_{r}^2\leq\abs{\nabla u}^2$
in order to merge the second and the third integrals in~\eqref{new_stab5}. In this way, for every $ p>1 $ we obtain
\begin{equation}\label{new_stab5_1}
\begin{split}
(n-p-a)&\int_{B_{1/4}}r^{-a}\abs{\nabla u}^p\,dx+\left(pa-\frac{a^2}{4}\max{\{1,p-1\}}\right)\int_{B_{1/4}}r^{-a}\abs{\nabla u}^{p-2}u_{r}^2\,dx
\\
&\leq C\int_{B_{1/2}\setminus B_{1/4}}\abs{\nabla u}^p\,dx.
\end{split}
\end{equation}

From the lower bound for $a$ in \eqref{a_condition} we deduce that the constant in front of the second integral in \eqref{new_stab5_1} is negative, and thus we can bound the left-hand side of \eqref{new_stab5_1} from below by
\[
\left(n-p-a+pa-\frac{a^2}{4}\max{\{1,p-1\}}\right)\int_{B_{1/4}}r^{-a}\abs{\nabla u}^p\,dx.
\]
Since the constant in front of the above integral is positive --- thanks now to the upper bound for $a$ in \eqref{a_condition} ---
we deduce that 
\begin{equation*}
\int_{B_{1/4}}r^{-a}\abs{\nabla u}^p\,dx\leq C \int_{B_{1/2}\setminus B_{1/4}}\abs{\nabla u}^p\,dx
\leq C \norm{\nabla u}^p_{L^p(B_1)},
\end{equation*}
where $C$ depends only on $n$, $p$, and $a$.

Applying this estimate (rescaled) to $ u_y(x):=u(x+y) $ with $ y\in \overline{B}_{1/4}$, since $B_{3/4}(y)\subset B_1$ we obtain for all $\rho\in(0,\frac14\cdot\frac34)$
\begin{equation}\label{mor_est}
\rho^{-a}\int_{B_\rho(y)}\abs{\nabla u}^p\,dx
\leq
\int_{B_\rho(y)}\abs{x-y}^{-a}\abs{\nabla u}^p\,dx\leq C\norm{\nabla u}^p_{L^p(B_{1})}.
\end{equation}
This proves that $ \nabla u\in M^{p,\beta}(B_{1/4}) $ for every $ \beta:=n-a>\overline{\beta} $, where $\overline{\beta}$ was defined in~\eqref{beta_bar}.

The fact that $\overline{\beta}>p$ follows from
	\begin{equation*}
	p< n-2-2\sqrt{\frac{n-1}{p-1}}\qquad\text{if}\,\,\,p\geq2\,\,\,\text{and}\,\,\,n > p+\frac{4p}{p-1}
	\end{equation*}
	and
	\begin{equation*}
	p< n-2(p-1)-2\sqrt{(p-1)^2+n-p}\qquad\text{if}\,\,\,p\in(1,2)\,\,\,\text{and}\,\,\,	n> 5p.
	\end{equation*}	

Therefore, after cutting-off $u$ outside of $B_{1/8}$ to have compact support in $B_{1/4}$, we can apply Proposition 3.1 and Theorems 3.1 and 3.2 in \cite{A} (see also the proof in \cite[Section 4]{CCh}) to deduce that $ u\in M^{\frac{p\beta}{\beta-p},\beta}(B_{1/8})$ for every $ \beta\in(\overline{\beta},n) $. 
Note that $M^{\frac{p\beta}{\beta-p},\beta}(B_{1/8}) = M^{p+\frac{p^2}{\beta-p},\beta}(B_{1/8})  $ and, thus, the first index is decreasing in $\beta$. 
Observe also that, thanks to \eqref{mor_est}, we control both the norms of $u$ in $M^{p+\frac{p^2}{\beta-p},\beta}(B_{1/8})$ and $\nabla u$ in $M^{p,\beta}(B_{1/8})$ in terms of $\norm{\nabla u}_{L^p(B_1)}$. This estimate in $B_{1/8}$ can be stated in $B_{1/4}$, as in \eqref{morrey_estimate_big_n}, using a scaling and covering argument.

Now that we have established \eqref{morrey_estimate_big_n}, if we further assume that $f$ is nonnegative, then we can replace $\norm{\nabla u}_{L^p(B_1)}$ by $\norm{u}_{L^1(B_1)}$ in the right-hand side of \eqref{morrey_estimate_big_n}. This can be done using~\eqref{L1controls_p} (properly rescaled) and a covering argument.
\end{proof}

Finally, we give the:

\begin{proof}[Proof of Theorem \ref{thm_big_n}]
	Estimate \eqref{estimate_big_n} follows from Theorem \ref{thm_big_n_morrey} by taking $q=p+\frac{p^2}{\beta-p}$ and using a covering argument. Observe that $q$ is defined as a decreasing function of~$\beta$, thus the upper bound for $q$ in \eqref{estimate_big_n} is given by $q_n=p+\frac{p^2}{\overline{\beta}-p}$.
	
	As in the proof of Theorem \ref{thm_big_n_morrey}, if we further assume $f$ to be nonnegative, then we can replace $\norm{\nabla u}_{L^p(B_1)}$ by $\norm{u}_{L^1(B_1)}$ in the right-hand side of \eqref{estimate_big_n}. This can be done using \eqref{L1controls_p} (properly rescaled) and a covering argument.
\end{proof}

\appendix
\section{Technical lemmata}\label{appendix_lemmata}
The first lemma of this section is a simple estimate for the $ L^1 $-norm of $\plaplacian u$. We assume that $u$ is a regular solution of \eqref{plap_eq} in $B_R$ and that the nonlinearity $f$ is nonnegative. Since we use this bound several times throughout the paper, we report it here together with its elementary proof.
\begin{lemma}\label{positive_plap}
	Let $u$ be a regular solution of $-\Delta_p u=f(u)$ in $B_R\subset\R^n$, 
	with $f\in C^1(\R)$ nonnegative. Then, for every $ \rho\in(0,R) $, we have
	\[
	\int_{B_\rho}\abs{\Delta_p u}\,dx\leq\frac{C}{R-\rho}\int_{B_R}\abs{\nabla u}^{p-1}\,dx,
	\]
	where $ C>0 $ is a constant depending only on $n$.
\end{lemma}
\begin{proof}
	First, we recall that since $\nabla u\in W^{1,2}_{\sigma,\text{loc}}(B_R) $ --- see Remark \ref{rmk_integrability} --- we can take the distributional divergence of $\abs{\nabla u}^{p-2}\nabla u$.
	
	We choose a nonnegative function $ \eta\in C_c^1(B_R)$ such that $ \eta\equiv1 $ in $B_\rho$ and $ \abs{\nabla\eta}\leq\frac{C}{R-\rho} $. Then, since $ -\Delta_p u\geq0 $, we have
	\begin{equation*}
	\begin{split}
	\int_{B_R}\abs{\Delta_p u}\eta\,dx
	&=-\int_{B_R}\diver\left(\abs{\nabla u}^{p-2}\nabla u\right)\eta\,dx
	=\int_{B_R}\abs{\nabla u}^{p-2}\nabla u\cdot\nabla\eta\,dx
	\\
	&\leq\frac{C}{R-\rho}\int_{B_R}\abs{\nabla u}^{p-1}\,dx,
	\end{split}
	\end{equation*}
	and this concludes the proof.
\end{proof}

In Section \ref{sec_doubling}, we also need an elementary Rellich-type criterion to establish strong convergence in $L^{p-1}(B_1)$ whenever $p\geq2$. This result will be applied with $w_k$ in its statement being a partial derivative of a stable solution $u_k$.
\begin{lemma}\label{lemma_convergence}
	Let $p\geq2$ and $(w_k)_k$ be a sequence of weakly differentiable functions in $L^{p-1}(B_1)$ such that
	\[
	\norm{w_k}_{L^{p-1}(B_1)}\leq C \qquad \text{and} \qquad  \norm{\abs{w_k}^{p-2}\nabla w_k}_{L^1(B_1)}\leq C
	\]
	for some constant $C$ independent of $k$. 
	
	Then, there exists a subsequence of $(w_k)_k$ that converges strongly in $L^{p-1}(B_1)$.
\end{lemma}
\begin{proof}
	For $\widetilde{w}_k:=\abs{w_k}^{p-2}w_k$ we have that
	\[
	\norm{\widetilde{w}_k}_{L^{1}(B_1)}\leq C \qquad \text{and} \qquad  \norm{\nabla \widetilde{w}_k}_{L^1(B_1)}\leq C,
	\]
	where $C$ does not depend on $k$. Thus, the sequence $(\widetilde{w}_k)_k$ is bounded in $W^{1,1}(B_1)$ and from the compactness of the immersion $ W^{1,1}(B_{1})\hookrightarrow L^1(B_{1}) $ we obtain that, up to a subsequence,
	\[
	\widetilde{w}_k\longrightarrow \widetilde{w} \qquad \text{strongly in}\,\,L^{1}(B_{1}),
	\]
	for some $\widetilde{w}\in L^1(B_1)$.
	
	Now, defining $w:=\abs{\widetilde{w}}^{p'-2}\widetilde{w}$, we have that
	\[
	\abs{w_k-w}^{p-1} \leq C_{n,p} \abs{\abs{w_k}^{p-2}w_k - \abs{w}^{p-2}w	} = C_{n,p} \abs{\widetilde{w}_k-\widetilde{w}},
	\]
	where we used that for $p\geq2$ we have $\abs{a-b}^p \leq C_{n,p}(\abs{a}^{p-2}a-\abs{b}^{p-2}b) \cdot (a-b)$ if~$a$ and~$b$ belong to $\R^n$; see \cite[Chapter 1, Lemma 4.4]{DiBbook}.
	As a consequence, up to subsequences, from the strong convergence of $\widetilde{w}_k$ to $\widetilde{w}$ in $L^1(B_1)$ we deduce that
	\[
	w_k\longrightarrow w \qquad \text{strongly in}\,\,L^{p-1}(B_{1}).
	\]
	This finishes the proof.
\end{proof}

The following result is a new interpolation inequality adapted to the $p$-Laplacian which holds for $p\geq2$ and every nice function in $B_1\subset\R^n$, independently of its boundary values. It gives a control on the $L^{p-1}$-norm of $\nabla u$ in terms of the weighted $L^1$-norm of the second derivatives of $u$ plus the $L^{p-1}$-norm of the function. We assume some regularity hypotheses on the function which are fulfilled by every stable regular solutions to $-\plaplacian u=f(u)$ with $f\in C^1(\R)$, by Remark \ref{rmk_integrability}. Note that, as for the interpolation inequalities of \cite[Theorem~7.28]{GT}, its validity in $\R^n$ follows immediately once it is proved in dimension one.

\begin{proposition}\label{prop5.2}\footnote{\label{football}The previous and printed versions of this paper claimed the same interpolation inequality in balls instead of cubes, and this made the proof (and probably the statement) not to be correct. Indeed, proving inequality (A.5) in the previous versions contained a mistake: it fails when, given $\e \in (0,1)$, one has $|I|<<\e$ ---as it occurs later in the proof in $\R^n$ when 1d sections of the ball are very small compared to $\e$. The proof that we give here below, now in cubes, is correct. Despite this error, all other results of the previous and printed versions of the article remain correct since, within Step 1 of the proof of Theorem~1.1,  the interpolation inequality may be applied in small enough cubes covering the ball $B_{1/2}$ instead of applying it in the whole ball, as already written in the current version.}
Let $p\geq2$, $\e\in(0,1)$, $Q=(0,1)^n$ be the unit cube of $\R^n$, and $u\in(C^{1}\cap W^{1,p-1})(Q)$ satisfy $\nabla u\in W^{1,1}_\sigma(Q)$ with $\sigma=\abs{\nabla u}^{p-2}$.

Then,
\begin{equation}\label{5.2}
\int_{B_{1}}\abs{\nabla u}^{p-1}\,dx \leq C\left( \e \int_{B_{1}}\abs{\nabla u}^{p-2}\lvert D^2u\rvert\,dx+ \e^{1-p}\int_{B_{1}}\abs{u}^{p-1}\,dx\right),
\end{equation}
where $C$ depends only on $n$ and $p$.
\end{proposition}

\begin{proof}
First, we observe that the regularity assumptions on $u$ ensure that all the integrals in~\eqref{5.2} are well defined.

We first prove \eqref{5.2} in dimension $n=1$. 
Given $\e>0$, let $u$ have the regularity assumed in the statement with $B_1$ replaced by $(0,\e)$. 
We claim that  
\begin{equation}\label{5.2claim1}
\inf_{(0,\e)}\abs{u'}^{p-1}\leq C \e^{-p}\int_{0}^{\e}\abs{u}^{p-1}\,dx,
\end{equation}
where the constant $C>0$ depends only on $p$.

The inequality is trivial if $\inf_{(0,\e)}\abs{u'}=0$. Therefore, we assume that $\inf_{(0,\e)}\abs{u'}>0$ and, by Bolzano's theorem, either $u'>0$ in $(0,\e)$ or $u'<0$ in $(0,\e)$. By eventually changing $u$ to $-u$, we may assume that $u'>0$ in $(0,\e)$.
For $a, b\in\R$ satisfying $0<a<\frac\e4<\frac{3\e}4<b<\e$, we have
\[
\frac\e2\inf_{(0,\e)}u'\leq(b-a)\inf_{(0,\e)}u'\leq(b-a)\inf_{(a,b)}u'\leq\int_{a}^{b}u'\,dx=u(b)-u(a).
\]
Integrating this inequality in $b\in(\frac{3\e}4,\e)$ we get
\[
\frac{\e^2}{8}\inf_{(0,\e)}u'\leq\int_{\frac{3\e}{4}}^{\e}u\,dx-\frac{\e}{4}u(a) 
\leq\int_{0}^{\e}\abs{u}\,dx-\frac{\e}{4}u(a).
\]
Integrating now in $a\in(0,\frac\e4)$,
\[
\frac{\e^3}{32}\inf_{(0,\e)}u'
\leq\frac\e4\int_{0}^{\e}\abs{u}\,dx-\frac{\e}{4}\int_{0}^{\frac\e4}u\,dx\leq\frac\e2\int_{0}^{\e}\abs{u}\,dx.
\]
Thus
\[
\inf_{(0,\e)}\abs{u'}=\inf_{(0,\e)}u'\leq\frac{16}{\e^2}\int_{0}^{\e}\abs{u}\,dx,
\]
and raising this inequality to $p-1\geq1$ we get
\begin{equation*}
\begin{split}
\inf_{(0,\e)}\abs{u'}^{p-1}&\leq C\e^{-2(p-1)}\left(\int_{0}^\e\abs{u}\,dx\right)^{p-1}
\leq C\e^{-p}\int_{0}^\e\abs{u}^{p-1}\,dx,
\end{split}
\end{equation*}
where $C$ depends only on $p$. This proves \eqref{5.2claim1}.

Therefore, by \eqref{5.2claim1}, there exists a point $x_0\in(0,\e)$ such that 
\begin{equation}\label{5.2_1}
\abs{u'(x_0)}^{p-1}\leq 2C \e^{-p}\int_{0}^{\e}\abs{u}^{p-1}\,dx.
\end{equation}
Let $x\in(0,\e)$ and $I_0\subset(0,\e)$ be the interval with end points $x_0$ and $x$. Then
\[
\abs{\abs{u'(x)}^{p-1}-\abs{u'(x_0)}^{p-1}} = \abs{\int_{I_0}\left(\abs{u'}^{p-1}\right)'\,dx},
\] 
and hence
\[
\abs{u'(x)}^{p-1} \leq (p-1)\int_{0}^\e\abs{u'}^{p-2}\abs{u''}\,dx+\abs{u'(x_0)}^{p-1}.
\] 
Combining this inequality with \eqref{5.2_1} and integrating in $x\in(0,\e)$, we obtain
\begin{equation}\label{5.2claim2}
\int_{0}^{\e}\abs{u'}^{p-1}\,dx \leq (p-1)\e\int_{0}^\e\abs{u'}^{p-2}\abs{u''}\,dx+C \e^{1-p}\int_{0}^{\e}\abs{u}^{p-1}\,dx.
\end{equation}

Note that we have not required any specific boundary values for $u$. Now, for any integer  $k>1$ we divide $(0,1)$ into $k$ disjoint intervals of length $\e:=1/k$. Now, using \eqref{5.2claim2} in each of these intervals of length $\e:=1/k$ and adding up all the inequalities given by~\eqref{5.2claim2}, we deduce
\begin{equation}\label{5.2claim3}
\int_{0}^1\abs{u'}^{p-1}\,dx \leq\frac{p-1}{k}\int_{0}^1\abs{u'}^{p-2}\abs{u''}\,dx+C k^{p-1}\int_{0}^1\abs{u}^{p-1}\,dx
\end{equation}
Since $\e\in (0,1)$, there exists an integer $k>1$ such that $\frac{1}{\e}\le k < \frac{2}{\e}$. This establishes the proposition in dimension one --- after replacing $\e$ by $\e/(p-1)$.

Finally, for $u$ defined in $(0,1)^n$,  denote $x=(x_1,x')\in\R\times\R^{n-1}$. Using \eqref{5.2claim3} with $n=1$ for every $x'$, we get
\begin{equation*}
\begin{split}
\int_{Q}\abs{u_{x_1}}^{p-1}\,dx 
&=\int_{(0,1)^{n-1}}\,dx'\int_0^1\,dx_1\abs{u_{x_1}(x)}^{p-1}
\\ 
&\leq C\, \e\int_{(0,1)^{n-1}}\,dx'\int_0^1\,dx_1\abs{u_{x_1}(x)}^{p-2}\abs{u_{x_1x_1}(x)} \\
&\hspace{2cm} +C\, \e^{1-p}\int_{(0,1)^{n-1}}\,dx'\int_0^1\,dx_1\abs{u(x)}^{p-1}
\\
&= C\left(\e\int_{Q}\abs{u_{x_1}(x)}^{p-2}\abs{u_{x_1x_1}(x)}\,dx+ \e^{1-p}\int_{Q}\abs{u(x)}^{p-1}\,dx\right).
\end{split}
\end{equation*}
Since the same inequality holds for the partial derivatives with respect to each variable $x_k$ instead of $x_1$, adding up the inequalities we conclude \eqref{5.2}. 
\end{proof}

We conclude this appendix with the statement of a general abstract lemma due to Simon~\cite{Simon} --- see also \cite[Lemma 3.1]{CSV}. It is extremely useful to ``absorb errors'' in larger balls of quantities controlled in smaller balls.

\begin{lemma}\label{lemma_simon}
Let $\beta\in \R$ and $C_0>0$. Let $\mathcal S: \mathcal B \rightarrow [0,+\infty]$ be a nonnegative function defined on the class $\mathcal B$  of open balls $B\subset \R^n$ and satisfying the following subadditivity property:
\[ \mbox{if}\quad  B \subset \bigcup_{j=1}^N B_j \quad  \mbox{ then }\quad \mathcal S(B)\le \sum_{j=1}^N \mathcal S(B_j). \]
Assume also that $\mathcal S(B_1)< \infty$.

Then, there exists $\delta$, depending only on $n$ and $\beta$, such that if
\begin{equation*}\label{hp-lem}
\rho^\beta \mathcal S\bigl(B_{\rho/4}(y)\bigr) \le \delta \rho^\beta \mathcal S\bigl(B_\rho(y)\bigr)+ C_0\quad \mbox{whenever }B_\rho(y)\subset B_1,
\end{equation*}
then
\[ \mathcal S(B_{1/2}) \le CC_0,\]
where $C$ depends only on $n$ and $\beta$.
\end{lemma}

\section{An alternative proof of the higher integrability result (Lemma \ref{lemma_Lpgamma}) using the Michael-Simon and Allard inequality.}\label{appendix_alt_proof}
We present here an alternative proof of Lemma \ref{lemma_Lpgamma} based on using the Sobolev inequality of Michael-Simon and Allard on the level sets of a stable regular solution $u$. It gives a control on the $L^{p+\gamma}$-norm of the gradient of $u$ in terms of its $L^p$-norm in a larger ball, with $\gamma=\frac{2(p-1)}{n-1}$ when $n\geq4$ and $\gamma=\frac{2(p-1)}{3}$ when $n\leq3$ --- see Remark \ref{rmk_gamma} for a comparison with the values of~$\gamma$ given by our other proof.

We first recall the celebrated Michael-Simon and Allard inequality, in the form presented in \cite{MS} --- see also \cite[Section 2]{CM} for a quick and easy-to-read proof of this important result. As pointed out in the beginning of the Introduction of \cite{MS}, as well as in its Example 3, the inequality holds in $C^2$ hypersurfaces. This will be useful for our purposes. 

\begin{theorem}[Michael-Simon \cite{MS} and Allard \cite{All} inequality]
	\label{thm_MS&A} Let $M$ be a $C^2$ $(n-1)$-dimen-\-sional hypersurface of $\,\R^{n}$, $ q\in[1,n-1) $, and $\varphi\in C^{1}(M)$ have compact support in $M$. If $M$ is compact without boundary, any function $\varphi\in C^1(M)$ is allowed.
	
	Then, there exists a positive constant $C$ depending only on $n$ and $q$, such that
	\begin{equation}\label{ms&a}
	\lVert \varphi\rVert_{L^{q^*}(M)}^q\leq C\int_M\big\{\abs{\nabla_T\varphi}^q+\abs{H\varphi}^q\big\}\,d\mathcal{H}^{n-1},
	\end{equation}
	where $ q^*=\frac{(n-1)q}{n-1-q}$ is the Sobolev exponent, $\nabla_T$ denotes the tangential gradient to $M$, and $H$ is the mean curvature of $M$, i.e., the sum of its $n-1$ principal curvatures.
\end{theorem}

\begin{proof}[Alternative proof of Lemma  \ref{lemma_Lpgamma}] \textbf{Step 1.}
	If $n\leq3$, we can add additional artificial variables and reduce to the case $n>3$, by Remark \ref{artificial}. We thus assume that $n>3$ and claim that
	\begin{equation}\label{12}
	\int_\R\,dt\left(\int_{\{u=t\}\cap B_{1/2}}\abs{\nabla u}^\frac{(p-1)(n-1)}{n-3}\,d\mathcal{H}^{n-1}\right)^\frac{n-3}{n-1}\leq C\norm{\nabla u}_{L^p(B_1)}^p,
	\end{equation}
	where $C$ depends only on $n$ and $p$. 
	
	In order to prove \eqref{12}, observe that the surface integral in the left-hand side of \eqref{12} can be equivalently computed over $\{u=t\}\cap B_{1/2}$ or over $\{u=t\}\cap B_{1/2}\cap\{\abs{\nabla u}>0\}$.
	We use the Michael-Simon and Allard inequality \eqref{ms&a} on the $C^2$ hypersurface\footnote{Recall that $u$ is $C^2$ in $\{\abs{\nabla u}>0\}$ since here the equation is uniformly elliptic; see Remark \ref{rmk_integrability}.} $M=\{u=t\}\cap\{\abs{\nabla u}>0\}$. We would like to apply the inequality to the function 
	$$\varphi=\abs{\nabla u}^\frac{p-1}{2}\eta,$$ 
	where $ \eta\in C_c^\infty(B_{3/4})$ satisfies $0\leq\eta\leq1$ and $\eta\equiv1$ in $B_{1/2}$; however, this function will not have, in general, compact support in $M=\{u=t\}\cap\{\abs{\nabla u}>0\}$. Thus,  in \eqref{ms&a} we take the test function
	\[
	\varphi_\e=\abs{\nabla u}^\frac{p-1}{2}\eta\,\phi(\abs{\nabla u}/\e),
	\]
	with $\e>0$, where $\phi\in C^\infty(\R)$ takes values into $[0,1]$, $\phi(t)=1$ if $t\geq2$, and $\phi(t)=0$ if $t\leq1$. Note that $\phi(\abs{\nabla u}/\e)$ is the cut-off function that we already used in the proof of Lemma \ref{lemma_stab1}. Using Fatou's lemma, \eqref{ms&a} with $q=2<n-1$, and the coarea formula, we obtain
	\begin{equation}\label{11}
	\begin{split}
	\int_\R\,dt\left(\int_{\{u=t\}\cap B_{1}}\varphi^\frac{2(n-1)}{n-3}\,d\mathcal{H}^{n-1}\right)^\frac{n-3}{n-1}
	&\leq\liminf_{\e\downarrow0}\int_\R\,dt\left(\int_{\{u=t\}\cap B_{1}}\varphi_\e^\frac{2(n-1)}{n-3}\,d\mathcal{H}^{n-1}\right)^\frac{n-3}{n-1}
	\\
	&\hspace{-3cm}\leq
	C\liminf_{\e\downarrow0}\int_\R\,dt\int_{\{u=t\}\cap B_{1}\cap\{\abs{\nabla u}>0\}}\left(\abs{\nabla_T\varphi_\e}^2+H^2\varphi_\e^2\right)\,d\mathcal{H}^{n-1}
	\\
	&\hspace{-3cm}=
	C\liminf_{\e\downarrow0}\int_{B_{1}\cap\{\abs{\nabla u}>0\}}\abs{\nabla u}\left(\abs{\nabla_T\varphi_\e}^2+H^2\varphi_\e^2\right)\,dx.
	\end{split}
	\end{equation}	
	Here the tangential gradient $\nabla_T$ and the mean curvature $H$ are referred to the level set of~$u$ passing through a given point $x$. 	
	For the square modulus of the tangential gradient of $\varphi_\e$, the Cauchy-Schwarz inequality gives
	\begin{equation*}
	\begin{split}
	\abs{\nabla_T\varphi_\e}^2
	&\leq
	\frac{3(p-1)^2}{4}\abs{\nabla u}^{p-3}\abs{\nabla_T\abs{\nabla u}}^2\eta^2 \phi^2(\abs{\nabla u}/\e)
	+3\abs{\nabla u}^{p-1}\abs{\nabla_T\eta}^2\phi^2(\abs{\nabla u}/\e)
	\\
	&\hspace{5cm}+3\abs{\nabla u}^{p-1}\eta^2\abs{\nabla_T\phi(\abs{\nabla u}/\e)}^2.
	\end{split}
	\end{equation*}
	Plugging this inequality in \eqref{11}, using $\abs{\nabla_T\phi(\abs{\nabla u}/\e)}^2\leq\abs{\nabla\phi(\abs{\nabla u}/\e)}^2$, \eqref{cutoff_2}, the fact that $\abs{\phi'(\abs{\nabla u}/\e)}$ is bounded independently of $\e$ and it is supported in $\{\e\leq\abs{\nabla u}\leq2\e\}$, and that $\abs{\nabla u}^2\leq4\e^2$ in this set, we obtain
	\begin{equation*}
	\begin{split}
	&\hspace{-4mm}\int_\R\,dt\left(\int_{\{u=t\}\cap B_{1}}\left(\abs{\nabla u}^\frac{p-1}{2}\eta\right)^\frac{2(n-1)}{n-3}\,d\mathcal{H}^{n-1}\right)^\frac{n-3}{n-1}
	\\
	&\hspace{4mm}\leq
	C\limsup_{\e\downarrow0}\int_{B_{1}\cap\{\abs{\nabla u}>0\}}\bigg\{\frac{3(p-1)^2}{4}\abs{\nabla u}^{p-2}\abs{\nabla_T\abs{\nabla u}}^2\eta^2
	+3\abs{\nabla u}^{p}\abs{\nabla_T\eta}^2
	\\
	&\hspace{12mm}+
	H^2\abs{\nabla u}^{p}\eta^2\bigg\}\phi^2(\abs{\nabla u}/\e)\,dx
	+C\limsup_{\e\downarrow0}\int_{B_{1}\cap\{\e<\abs{\nabla u}<2\e\}}
	\abs{\nabla u}^{p-2}\lvert D^2u\rvert^2\eta^2\,dx.
	\end{split}
	\end{equation*}
	Now, using \eqref{funct_space1} and dominated convergence we deduce that the last term is zero. Using that $\phi^2\leq 1$ in
	the first integral in the right-hand side of the inequality, we get
	\begin{equation*}
	\begin{split}
	&\hspace{-0.2cm}\int_\R\,dt\left(\int_{\{u=t\}\cap B_{1}}\left(\abs{\nabla u}^\frac{p-1}{2}\eta\right)^\frac{2(n-1)}{n-3}\,d\mathcal{H}^{n-1}\right)^\frac{n-3}{n-1}
	\\ 
	&\hspace{1.5cm}\leq 
	C\int_{B_{1}\cap\{\abs{\nabla u}>0\}}\bigg\{\frac{3(p-1)^2}{4}\abs{\nabla u}^{p-2}\abs{\nabla_T\abs{\nabla u}}^2
	+H^2\abs{\nabla u}^{p}\bigg\}\eta^2\,dx
	\\
	&\hspace{6.5cm}+C\int_{B_{1}}\abs{\nabla u}^{p} \abs{\nabla_T\eta}^2\,dx.
	\end{split}
	\end{equation*}	
	Since $H^2\leq(n-1)\abs{A}^2$, we deduce that
	\begin{equation*}
	\begin{split}
	&\int_\R\,dt\left(\int_{\{u=t\}\cap B_{1}}\left(\abs{\nabla u}^\frac{p-1}{2}\eta\right)^\frac{2(n-1)}{n-3}\,d\mathcal{H}^{n-1}\right)^\frac{n-3}{n-1}
	\\
	&\hspace{0.5cm}\leq C \int_{B_1\cap\{\abs{\nabla u}>0\}} \left\{(p-1)\abs{\nabla u}^{p-2}\abs{\nabla_T\abs{\nabla u}}^2+\abs{A}^2\abs{\nabla u}^p\right\}\eta^2\,dx
	+C\int_{B_1}\abs{\nabla u}^{p}\,dx.
	\end{split}
	\end{equation*}	
	Finally, we use the stability condition in its geometric form, Theorem \ref{thm_SZ_stability}, to obtain
	\begin{equation*}
	\begin{split}
	&\hspace{-0.6cm}\int_\R\,dt\left(\int_{\{u=t\}\cap B_{1}}\left(\abs{\nabla u}^\frac{p-1}{2}\eta\right)^\frac{2(n-1)}{n-3}\,d\mathcal{H}^{n-1}\right)^\frac{n-3}{n-1}
	\\
	&\hspace{1cm}\leq
	C \int_{B_1}\abs{\nabla u}^p\abs{\nabla \eta}^2\,dx +C\int_{B_1}\abs{\nabla u}^{p}\,dx
	\leq C \norm{\nabla u}^p_{L^p(B_1)}.
	\end{split}
	\end{equation*}
	This proves \eqref{12}. 
	
	\medskip
	
	\textbf{Step 2.} Now we show that, for almost every $t\in\R$,
	\begin{equation}\label{13}
	\int_{\{u=t\}\cap B_{1/2}}\abs{\nabla u}^{p-1}\,d\mathcal{H}^{n-1}
	\leq C\norm{\nabla u}_{L^{p-1}(B_1)}^{p-1}
	\leq C\norm{\nabla u}_{L^p(B_1)}^{p-1},
	\end{equation}
	where the constants $C$ depend only on $n$ and $p$. Note that this estimate is similar to \eqref{89}, where we assumed $\norm{\nabla u}_{L^p(B_1)}=1$, but we cannot deduce \eqref{13} from \eqref{89} through H\"older's inequality, since we do not have a control on the measure of $\{u=t\}\cap B_{1/2}$. However, the proofs of both estimates rely on the same idea.
	
	First, we take $ \eta $ as defined in Step 1 and we claim that, for almost every $t\in\R$,
	\begin{equation}\label{9}
	\int_{\{u=t\}\cap B_{1}}\abs{\nabla u}^{p-1}\eta\,d\mathcal{H}^{n-1}=-\int_{\{u>t\}\cap B_{1}}\diver\left(\eta\abs{\nabla u}^{p-2}\nabla u\right)\,dx.
	\end{equation}
	Observe that $u$ is not regular enough to apply Sard's theorem and deduce regularity of $\{u=t\}$ for a.e. $t$, and hence we cannot integrate by parts in the set $\{u>t\}\cap B_{1}$. However, as in the proof of~\eqref{4}, we can use a smooth approximation $K_\e(s)$ of the characteristic function of~$ \R_+ $, and then send $\e\downarrow0$, to establish \eqref{9}.
	
	From \eqref{9} we obtain that
	\begin{equation*}
	\int_{\{u=t\}\cap B_{1}}\abs{\nabla u}^{p-1}\eta\,d\mathcal{H}^{n-1} \leq
	\int_{B_{1}}\abs{\plaplacian u}\eta\,dx+\int_{B_{1}}\abs{\nabla u}^{p-2}\abs{\nabla u\cdot\nabla\eta}\,dx.
	\end{equation*}
	Now, since $0\leq\eta\leq1$ has compact support in $B_{3/4}$, we can use Lemma \ref{positive_plap} to obtain  
	\[
	\int_{B_1}\abs{\plaplacian u}\eta\,dx\leq 
	\int_{B_{3/4}}\abs{\plaplacian u}\,dx\leq
	C\norm{\nabla u}_{L^{p-1}(B_1)}^{p-1}.
	\]
	Combining this with the previous bound, we conclude the proof of \eqref{13}.
	
	\medskip
	
	\textbf{Step 3.}
	Finally, from \eqref{12} and \eqref{13} we deduce \eqref{19}. 
	For this, we use H\"older's inequality with exponents $ q=\frac{n-1}{n-3}$ and $q'=\frac{n-1}{2}$, as well as the coarea formula, to obtain
	\begin{equation*}
	\begin{split}
	&\hspace{-0.3cm}\int_{B_{1/2}}\abs{\nabla u}^{p+\frac{2(p-1)}{n-1}}\,dx
	=\int_\R\,dt\int_{\{u=t\}\cap B_{1/2}}\abs{\nabla u}^{p-1} \abs{\nabla u}^\frac{2(p-1)}{n-1}\,d\mathcal{H}^{n-1}
	\\ 
	&\hspace{0.8cm}\leq
	\int_\R\,dt	\left(\int_{\{u=t\}\cap B_{1/2}}\abs{\nabla u}^\frac{(p-1)(n-1)}{n-3}\,d\mathcal{H}^{n-1}\right)^\frac{n-3}{n-1} \left(\int_{\{u=t\}\cap B_{1/2}}\abs{\nabla u}^{p-1}\,d\mathcal{H}^{n-1}\right)^\frac2{n-1}
	\\
	&\hspace{0.8cm}\leq
	C\norm{\nabla u}_{L^p(B_1)}^{\frac{2(p-1)}{n-1}}
	\int_\R\,dt	\left(\int_{\{u=t\}\cap B_{1/2}}\abs{\nabla u}^\frac{(p-1)(n-1)}{n-3}\,d\mathcal{H}^{n-1}\right)^\frac{n-3}{n-1}
	\\
	&\hspace{0.8cm}\leq
	C\norm{\nabla u}_{L^p(B_1)}^{p+\frac{2(p-1)}{n-1}}.
	\end{split}
	\end{equation*}
	This concludes the alternative proof of Lemma \ref{lemma_Lpgamma}.
\end{proof}

	\addsection*{Acknowledgments} The authors would like to thank Lucio Boccardo and Luigi Orsina for a simplification in the proof of Lemma \ref{lemma_convergence}.

\end{document}